\documentclass[12pt]{article}%
\usepackage[T1]{fontenc}
\usepackage{textcomp}
\usepackage{cite}
\usepackage{amsmath}
\usepackage{amsfonts}
\usepackage{amssymb}
\usepackage{graphicx}
\usepackage{color,enumitem,graphicx}%
\providecommand{\U}[1]{\protect\rule{.1in}{.1in}}
\newtheorem{theorem}{Theorem}[section]

\newtheorem{corollary}[theorem]{Corollary}

\newtheorem{definition}[theorem]{Definition}

\newtheorem{lemma}[theorem]{Lemma}

\newtheorem{proposition}[theorem]{Proposition}
\newtheorem{remark}[theorem]{Remark}

\newenvironment{proof}[1][Proof]{\noindent\textbf{#1.} }{\ \rule{0.5em}{0.5em}}
\begin{document}

\title{Minimization of quotients with variable exponents}
\author{C.O. Alves$^{\text{\thinspace a}}$, M.D.H. Bola\~{n}os$^{\text{\thinspace b}}%
$, G. Ercole$^{\text{\thinspace b}}$ \thanks{Corresponding author}\\{\small {$^{\mathrm{a}}$} Universidade Federal de Campina Grande, Campina
Grande, PB, 58109-970, Brazil. }\\{\small E-mail: coalves@dme.ufcg.edu.br}\\{\small {$^{\mathrm{b}}$} Universidade Federal de Minas Gerais, Belo
Horizonte, MG, 30.123-970, Brazil.}\\{\small E-mail: emdi\_82@hotmail.com }\\{\small E-mail:grey@mat.ufmg.br}}
\maketitle

\begin{abstract}
Let $\Omega$ be a bounded domain of $\mathbb{R}^{N}$, $p\in C^{1}%
(\overline{\Omega}),$ $q\in C(\overline{\Omega})$ and $l,j\in\mathbb{N}.$ We
describe the asymptotic behavior of the minimizers of the Rayleigh quotient
$\frac{\Vert\nabla u\Vert_{lp(x)}}{\Vert u\Vert_{jq(x)}}$, first when
$j\rightarrow\infty$ and after when $l\rightarrow\infty.$

\end{abstract}

\noindent\textbf{2010 AMS Classification.} 35B40; 35J60; 35P30.


\section{Introduction}

Let $\Omega$ be a bounded domain of $\mathbb{R}^{N},$ $N\geq2,$ and consider
the Rayleigh quotient
\begin{equation}
\frac{\left\Vert \nabla u\right\Vert _{p(x)}}{\left\Vert u\right\Vert _{q(x)}%
}, \label{a}%
\end{equation}
associated with the immersion of the Sobolev space $W_{0}^{1,p(x)}(\Omega)$
into the Lebesgue space $L^{q(x)}(\Omega),$ where the variable exponents
satisfy%
\[
1<\inf_{\Omega}p(x)\leq\sup_{\Omega}p(x)<\infty
\]
and
\[
1<q(x)<p^{\ast}(x):=\left\{
\begin{array}
[c]{ccc}%
\frac{Np(x)}{N-p(x)} & \mathrm{if} & p(x)<N\\
\infty & \mathrm{if} & p(x)\geq N.
\end{array}
\right.
\]

In this paper we study the behavior of the least Rayleigh quotients when the
functions $p(x)$ and $q(x)$ become arbitrarily large. Our script is based on
the paper \cite{EP16}, where these functions are constants. Thus, in order to
overcome the difficulties imposed by the fact that the exponents depend on
$x$, we adapt arguments developed by Franzina and Lindqvist in \cite{FL13},
where $p(x)=q(x).$ Actually, our results in the present paper generalize those
of \cite{EP16} for variable exponents and complement the approach of
\cite{FL13}.

In \cite{EP16}, Ercole and Pereira first studied the behavior, when
$q\rightarrow\infty,$ of the positive minimizers $w_{q}$ corresponding to
\[
\lambda_{q}:=\min\left\{  \left\Vert \nabla u\right\Vert _{L^{p}(\Omega)}:u\in
W_{0}^{1,p}(\Omega)\quad\mathrm{in}\quad\left\Vert u\right\Vert _{L^{q}%
(\Omega)}=1\right\}  ,
\]
for a fixed $p>N.$ An $L^{\infty}$-normalized function $u_{p}\in W_{0}%
^{1,p}(\Omega)$ is obtained as the uniform limit in $\overline{\Omega}$ of a
sequence $w_{q_{n}},$ with $q_{n}\rightarrow\infty$. Such a function is
positive in $\Omega,$ assumes its maximum only at a point $x_{p}$ and
satisfies
\[
\left\{
\begin{array}
[c]{ll}%
-\Delta_{p}u=\Lambda_{p}\delta_{x_{p}} & \mathrm{in}\quad\Omega\\
u=0 & \mathrm{on}\quad\partial\Omega,
\end{array}
\right.
\]
where%
\[
\Lambda_{p}:=\min\left\{  \left\Vert \nabla u\right\Vert _{L^{p}(\Omega)}:u\in
W_{0}^{1,p}(\Omega)\quad\mathrm{in}\quad\left\Vert u\right\Vert _{L^{\infty
}(\Omega)}=1\right\}
\]
and $\delta_{x_{p}}$ denotes the Dirac delta distribution concentrated at
$x_{p}.$ In the sequence, the behavior of the pair $\left(  \Lambda_{p}%
,u_{p}\right)  ,$ as $p\rightarrow\infty,$ is determined. In fact, it is
proved that
\[
\lim_{p\rightarrow\infty}\Lambda_{p}=\Lambda_{\infty}:=\inf_{0\not \equiv v\in
W_{0}^{1,\infty}(\Omega)}\frac{\left\Vert \nabla v\right\Vert _{\infty}%
}{\left\Vert v\right\Vert _{\infty}}%
\]
and that there exist a sequence $p_{n}\rightarrow\infty,$ a point $x_{\ast}%
\in\Omega$ and a function $u_{\infty}\in W_{0}^{1,\infty}(\Omega)\cap
C(\overline{\Omega})$ such that: $x_{p_{n}}\rightarrow x_{\ast},$ $\left\Vert
\rho\right\Vert _{\infty}=\rho(x_{\ast}),$ $u_{\infty}\leq\frac{\rho
}{\left\Vert \rho\right\Vert _{\infty}}$ and $u_{p_{n}}\rightarrow u_{\infty
},$ uniformly in $\overline{\Omega}.$ Moreover, it is shown that: $u_{\infty}$
is also a minimizer of $\Lambda_{\infty},$ assumes its maximum value $1$ only
at $x_{\ast}$ and satisfies%
\[
\left\{
\begin{array}
[c]{ll}%
\Delta_{\infty}u=0 & \text{in }\Omega\backslash\left\{  x_{\ast}\right\} \\
u=\frac{\rho}{\left\Vert \rho\right\Vert _{\infty}} & \text{on }%
\partial\left(  \Omega\backslash\left\{  x_{\ast}\right\}  \right)  =\left\{
x_{\ast}\right\}  \cup\partial\Omega
\end{array}
\right.
\]
in the viscosity sense.

In \cite{FL13}, Franzina and Lindqvist determined the exact asymptotic
behavior, as $j\rightarrow\infty,$ of both the minimum $\Lambda_{jp(x)}$ of
the quotients $\frac{\left\Vert \nabla u\right\Vert _{jp(x)}}{\left\Vert
u\right\Vert _{jp(x)}}$ and its respective $jp(x)$-normalized minimizer
$u_{j}.$ It is proved that
\[
\lim_{j\rightarrow\infty}\Lambda_{jp(x)}=\Lambda_{\infty}%
\]
and that a subsequence of $\left(  u_{j}\right)  _{j\in\mathbb{N}}$ converges
uniformly in $\overline{\Omega}$ to a nonnegative function $0\not \equiv
u_{\infty}\in C(\overline{\Omega})\cap W_{0}^{1,\infty}(\Omega)$ satisfying,
in the viscosity sense, the equation%
\[
\max\left\{  \Lambda_{\infty}-\dfrac{\left\vert \nabla u\right\vert }{u}%
\quad,\quad\Delta_{\infty(x)}\left(  \dfrac{u}{\left\Vert \nabla u\right\Vert
_{\infty}}\right)  \right\}  =0\quad\mathrm{in}\quad\Omega,
\]
where the operator $\Delta_{\infty(x)}$ is defined by
\[
\Delta_{\infty(x)}u:=\Delta_{\infty}u+|\nabla u|^{2}\ln|\nabla u|\langle\nabla
u,\nabla\ln p\rangle.
\]

In the present paper we assume that $p\in C^{1}(\overline{\Omega})$, $q\in
C(\overline{\Omega})$ and $1\leq q(x)<p^{\ast}(x)$ in $\overline{\Omega}.$
After presenting, in Section 2, a brief review on the theory of
Sobolev-Lebesgue spaces with variable exponents, we show in Section 3 that
\[
\Lambda_{1}:=\inf\left\{  \frac{\left\Vert \nabla v\right\Vert _{p(x)}%
}{\left\Vert v\right\Vert _{q(x)}}:v\in W_{0}^{1,p(x)}(\Omega)\setminus
\left\{  0\right\}  \right\}  =\frac{\left\Vert \nabla u\right\Vert _{p(x)}%
}{\left\Vert u\right\Vert _{q(x)}}>0
\]
for some $u\in W_{0}^{1,p(x)}(\Omega)\setminus\left\{  0\right\}  .$ Moreover,
taking \cite{FL13} and \cite{PS14} as reference, we derive the following
Euler-Lagrange equation corresponding to this minimization problem%
\begin{equation}
-\operatorname{div}\left(  \left\vert \frac{\nabla u}{K(u)}\right\vert
^{p(x)-2}\frac{\nabla u}{K(u)}\right)  =\Lambda S(u)\left\vert \frac{u}%
{k(u)}\right\vert ^{q(x)-2}\frac{u}{k(u)}, \label{iia}%
\end{equation}
where $\Lambda=\Lambda_{1},$%
\begin{equation}
K(u):=\left\Vert \nabla u\right\Vert _{p(x)},\quad k(u):=\left\Vert
u\right\Vert _{q(x)},\quad\mathrm{and}\quad S(u):=\frac{\int_{\Omega
}\left\vert \frac{\nabla u}{K(u)}\right\vert ^{p(x)}\mathrm{d}x}{\int_{\Omega
}\left\vert \frac{u}{k(u)}\right\vert ^{q(x)}\mathrm{d}x}. \label{iia1}%
\end{equation}

We consider (\ref{iia})-(\ref{iia1}) as an eigenvalue problem. Thus, if a pair
$\left(  \Lambda,u\right)  \in\mathbb{R}\times W_{0}^{1,p(x)}(\Omega
)\setminus\left\{  0\right\}  $ solves (\ref{iia})-(\ref{iia1}) we say that
$\Lambda$ is an eigenvalue and $u$ is an eigenfunction corresponding to
$\Lambda.$ In this setting, $\Lambda_{1}$ is the first eigenvalue and any of
its corresponding eigenfunctions is a first eigenfunction. We show that any
first eigenfunction do not change sign in $\Omega$ and, for the sake of
completeness, we apply a minimax scheme based on Kranoselskii genus to obtain
an increasing and unbounded sequence of eigenvalues.

Our main results are established in Sections 4 and 5. First we consider a
natural $l>N$ and show, in Section 4, that%
\[
\mu_{l}:=\inf\left\{  \frac{\left\Vert \nabla v\right\Vert _{lp(x)}%
}{\left\Vert v\right\Vert _{\infty}}:v\in W_{0}^{1,lp(x)}(\Omega
)\setminus\{0\}\right\}  =\lim_{j\rightarrow\infty}\Lambda_{l,j},
\]
where%
\[
\Lambda_{l,j}:=\inf\left\{  \frac{\left\Vert \nabla v\right\Vert _{lp(x)}%
}{\left\Vert v\right\Vert _{jq(x)}}:v\in W_{0}^{1,lp(x)}(\Omega)\setminus
\{0\}\right\}  .
\]
Moreover, by using the results of Section 3, we argue that for each fixed
$j>1$ there exists a positive minimizer $u_{l,j}\in W_{0}^{1,lp(x)}%
(\Omega)\setminus\left\{  0\right\}  $ for $\Lambda_{l,j}.$ Hence, the
compactness of the embedding $W_{0}^{1,lp(x)}(\Omega)\hookrightarrow
C(\overline{\Omega})$ implies that $\mu_{l}$ is achieved at a function $w_{l}$
which is obtained as the uniform limit of $u_{l,j_{m}}$ for a subsequence
$j_{m}\rightarrow\infty.$

We also show in Section 4, by using arguments developed in \cite{HL16}, that
$\mu_{l}$ is achieved at $u$ if, and only if,%
\[
-\operatorname{div}\left(  \left\vert \frac{\nabla u}{K_{l}(u)}\right\vert
^{lp(x)-2}\frac{\nabla u}{K_{l}(u)}\right)  =\mu_{l}\left(  \int_{\Omega
}\left\vert \frac{\nabla u}{K_{l}(u)}\right\vert ^{lp(x)}\mathrm{d}x\right)
\operatorname{sgn}(u(x_{0}))\delta_{x_{0}},
\]
where $K_{l}(u)=\Vert\nabla u\Vert_{lp(x)}$ and $x_{0}$ is the only point
where $u$ reaches its uniform norm.

Finally, in Section 5, we study the asymptotic behavior of $\mu_{l}$ and of
its normalized extremal function $w_{l}$ ($\left\Vert w_{l}\right\Vert
_{\infty}=1$ and $\mu_{l}=\left\Vert \nabla w_{l}\right\Vert _{lp(x)}$), when
$l\rightarrow\infty$. We prove that
\[
\lim_{l\rightarrow\infty}\mu_{l}=\Lambda_{\infty}%
\]
and that there exist $l_{n}\rightarrow\infty,$ $x_{\star}\in\Omega$ and
$w_{\infty}\in W_{0}^{1,\infty}(\Omega)\cap C(\overline{\Omega})$ such that
$w_{l_{n}}\rightarrow w_{\infty}$ uniformly in $\overline{\Omega}$ and
\[
0\leq w_{\infty}\leq\frac{d}{\left\Vert d\right\Vert _{\infty}}\quad
\mathrm{a.e.}\,\Omega\quad\mathrm{and}\quad w_{\infty}(x_{\star}%
)=\frac{d(x_{\star})}{\left\Vert d\right\Vert _{\infty}},
\]
where $d(x)=\operatorname{dist}(x,\partial\Omega)$ is the function distance to
the boundary. It is well-known that $d\in W_{0}^{1,\infty}(\Omega)$ and
\[
\Lambda_{\infty}=\frac{1}{\left\Vert d\right\Vert _{\infty}}.
\]
Moreover, we prove that $\Lambda_{\infty}$ is attained at $w_{\infty}$ and
that this function satisfies%
\[
\left\{
\begin{array}
[c]{ll}%
\Delta_{\infty(x)}\left(  \dfrac{u}{\left\Vert \nabla u\right\Vert _{\infty}%
}\right)  =0 & \mathrm{in}\quad D=\Omega\setminus\{x_{\star}\}\\
\dfrac{u}{\left\Vert \nabla u\right\Vert _{\infty}}=d & \mathrm{on}%
\quad\partial D=\partial\Omega\cup\left\{  x_{\star}\right\}  .
\end{array}
\right.
\]
in the viscosity sense.

Due to the lack of a suitable version of the Harnack's inequality for the
"variable infinity operator" $\Delta_{\infty(x)},$ one cannot guarantee that
the function $w_{\infty}$ is strictly positive in $\Omega.$

At the end of Section 5, by using a uniqueness result proved in \cite{Li10}
for the equation $\Delta_{\infty(x)}u=0,$ we provide a sufficient condition on
$\Omega$ for the equality%
\[
w_{\infty}=\frac{d}{\left\Vert d\right\Vert _{\infty}}\quad\mathrm{in}%
\quad\overline{\Omega}%
\]
to hold.

After comparing our results with those of \cite{FL13}, it is interesting to
remark that the minimum of the quotients $\frac{\left\Vert \nabla u\right\Vert
_{lp(x)}}{\left\Vert u\right\Vert _{jq(x)}}$ converges to $\Lambda_{\infty}$
independently of how $lp(x)$ and $jq(x)$ go to $\infty:$ if either
$l=j\rightarrow\infty$ in the case $p(x)=q(x)$ or $j\rightarrow\infty$ firstly
and then $l\rightarrow\infty.$ However, the same do not hold for the
corresponding minimizers (or for their respective limit problems). The
distinction seems to be due to the Dirac delta that appears in the right-hand
term of the Euler-Lagrange equation (\ref{iia}) when $q(x)$ is replaced by
$jq(x)$ and $j$ is taken to infinity. The same distinction appears when $p$
and $q$ are constant, as one can check from \cite{EP16} and \cite{JLM}.

\section{Preliminaries}

In this section we recall some definitions and results on the Sobolev-Lebesgue
spaces with variable exponents.

Let $\Omega$ be a bounded domain in $\mathbb{R}^{N}$ and $p\in C(\overline
{\Omega})$ such that $1<p^{-}:=\inf p(x)\leq p^{+}:=\sup p(x)<\infty$. Let
$L^{p(x)}(\Omega)$ denote the space of the Lebesgue measurable functions
$u:\Omega\rightarrow\mathbb{R}$ such that
\[
\int_{\Omega}|u|^{p(x)}\mathrm{d}x<\infty,
\]
endowed with the Luxemburg norm
\begin{equation}
\left\Vert u\right\Vert _{p(x)}=\inf\left\{  \gamma>0:\int_{\Omega}\left\vert
\frac{u(x)}{\gamma}\right\vert ^{p(x)}\frac{\mathrm{d}x}{p(x)}\leq1\right\}  .
\label{normpx}%
\end{equation}
Note that (\ref{normpx}) is equivalent to the norm
\begin{equation}
|u|_{p(x)}=\inf\left\{  \gamma>0:\int_{\Omega}\left\vert \frac{u(x)}{\gamma
}\right\vert ^{p(x)}\mathrm{d}x\leq1\right\}  \label{e}%
\end{equation}
introduced by \cite{DH11} and \cite{FZ01}. In fact, we have
\[
\frac{1}{p^{+}}\left\vert u\right\vert _{p(x)}\leq\left\Vert u\right\Vert
_{p(x)}\leq\left\vert u\right\vert _{p(x)},\quad\forall\,u\in L^{p(x)}%
(\Omega).
\]

An important concept in the theory of spaces $L^{p(x)}(\Omega)$ is the modular function.

\begin{definition}
The function $\rho:L^{p(x)}(\Omega)\rightarrow\mathbb{R}$ defined by%
\[
\rho(u)=\int_{\Omega}|u|^{p(x)}\frac{\mathrm{d}x}{p(x)},
\]
is called the modular function associated to the space $L^{p(x)}(\Omega)$.
\end{definition}

The following proposition lists some properties of the modular function .

\begin{proposition}
\label{modul} Let $u\in L^{p(x)}(\Omega)\setminus\{0\}$, then

\begin{enumerate}
\item[a)] $\Vert u\Vert_{p(x)}=a$ if, and only if, $\rho(\frac{u}{a})=1;$

\item[b)] $\Vert u\Vert_{p(x)}<1$ $(=1;>1)$ if, and only if, $\rho(u)<1$
$(=1;>1);$

\item[c)] If $\Vert u\Vert_{p(x)}>1$ then $\Vert u\Vert_{p(x)}<\Vert
u\Vert_{p(x)}^{p^{-}}\leq\rho(u)\leq\Vert u\Vert_{p(x)}^{p^{+}};$

\item[d)] If $\Vert u\Vert_{p(x)}<1$ then $\Vert u\Vert_{p(x)}^{p^{+}}\leq
\rho(u)\leq\Vert u\Vert_{p(x)}^{p^{-}}<\Vert u\Vert_{p(x)}.$
\end{enumerate}
\end{proposition}

For a posterior use, we recall the following estimate valid for an arbitrary
$u\in L^{\infty}(\Omega)$:%

\begin{equation}
\left\Vert u\right\Vert _{p(x)}\leq\alpha\left\Vert u\right\Vert _{\infty
},\quad\mathrm{where}\quad\alpha:=\left\{
\begin{array}
[c]{ccc}%
\left\vert \Omega\right\vert ^{1/p^{+}} & \mathrm{if} & \left\vert
\Omega\right\vert \leq1,\\
\left\vert \Omega\right\vert ^{1/p^{-}} & \mathrm{if} & \left\vert
\Omega\right\vert >1.
\end{array}
\right.  \label{infi2}%
\end{equation}
This estimate is easily verified by applying item $\mathrm{b)}$ of Proposition
\ref{modul} to the function $\frac{u}{\alpha\left\Vert u\right\Vert _{\infty}%
}.$

We define the Sobolev space%
\[
W^{1,p(x)}(\Omega):=\{u\in L^{p(x)}(\Omega):\left\vert \nabla u\right\vert \in
L^{p(x)}(\Omega)\},
\]
endowed with the norm%
\[
\Vert u\Vert_{1,p(x)}:=\Vert u\Vert_{p(x)}+\Vert\nabla u\Vert_{p(x)}.
\]
Both $(L^{p(x)}(\Omega),\Vert.\Vert_{p(x)})$ and $(W^{1,p(x)}(\Omega
),\Vert.\Vert_{1,p(x)})$ are separable and uniformly convex (therefore,
reflexive) Banach spaces.

The Sobolev space $W_{0}^{1,p(x)}(\Omega)$ is defined as the closure of
$C_{0}^{\infty}(\Omega)$ in $W^{1,p(x)}(\Omega)$. In this space, $\Vert
\nabla\cdot\Vert_{p(x)}$ is a norm equivalent to norm $\Vert\cdot
\Vert_{1,p(x)}$ and this is a consequence of the following proposition.

\begin{proposition}
(see \cite{FZ01}) Let $p\in C(\overline{\Omega})$ with $p^{-}>1.$ There exists
a positive constant $C$ such that%
\[
\Vert u\Vert_{p(x)}\leq C\Vert\nabla u\Vert_{p(x)},\quad\forall\,u\in
W_{0}^{1,p(x)}(\Omega).
\]

\end{proposition}

Now, we recall some facts involving exponents $q(x)\leq p(x).$

\begin{proposition}
(see \cite{FZ01})\label{imer} Let $p,q\in C(\overline{\Omega})$. Then%
\[
L^{p(x)}(\Omega)\subset L^{q(x)}(\Omega)
\]
if, and only if, $q(x)\leq p(x)$ in $\Omega$. Additionally, the embedding is continuous.
\end{proposition}

From now on, the notation $f\ll g$ will mean that $f(x)\leq g(x)$ for all
$x\in\overline{\Omega}$ and%
\[
\inf\limits_{\Omega}\left(  g(x)-f(x)\right)  >0.
\]

\begin{proposition}
(\cite{FS01}, \cite{FZ01})\label{pa} Let $p,q\in C(\overline{\Omega})$ and
$1\leq q(x)\leq p^{\ast}(x)$ in $\overline{\Omega}$. The embedding
\[
W_{0}^{1,p(x)}(\Omega)\hookrightarrow L^{q(x)}(\Omega)
\]
is continuous. Moreover, it is compact whenever $q\ll p^{\star}.$
\end{proposition}

We define the operator $p(x)$-Laplacian by $\Delta_{p(x)}u:=\operatorname{div}%
(|\nabla u|^{p(x)-2}\nabla u)$ and consider the Dirichlet problem%

\begin{equation}
\left\{
\begin{array}
[c]{ll}%
-\Delta_{p(x)}u=f(x,u), & x\in\Omega\\
u=0 & x\in\partial\Omega
\end{array}
\right.  \label{prf}%
\end{equation}
where $f\in C(\overline{\Omega}\times\mathbb{R},\mathbb{R}).$

We say that a function $u\in W_{0}^{1,p(x)}(\Omega)$ is a weak solution of
(\ref{prf}) if, and only if,
\[
\int_{\Omega}\left\vert \nabla u\right\vert ^{p(x)-2}\nabla u\cdot\nabla
\eta\,\mathrm{d}x=%
{\displaystyle\int_{\Omega}}
f(x,u)\eta\mathrm{d}x,\quad\forall\,\eta\in W_{0}^{1,p(x)}(\Omega).
\]

\begin{proposition}
\label{sol}Weak solutions of (\ref{prf}) belong to $L^{\infty}(\Omega)$
provided that $f$ satisfies the sub-critical growth condition
\[
\left\vert f(x,t)\right\vert \leq c_{1}+c_{2}\left\vert t\right\vert
^{\alpha(x)-1},\quad\forall\,(x,t)\in\Omega\times\mathbb{R},
\]
where $\alpha\in C(\overline{\Omega})$ and $1<\alpha\ll p^{\ast}.$
\end{proposition}

\begin{proposition}
\label{solab}Suppose that $p(x)$ is H\"{o}lder continuous on $\overline
{\Omega}.$ If $u\in W_{0}^{1,p(x)}(\Omega)\cap L^{\infty}(\Omega)$ is a weak
solution of (\ref{prf}), then $u\in C^{1,\tau}(\overline{\Omega})$ for some
$\tau\in(0,1).$
\end{proposition}

The following strong maximum principle for $p(x)$-Laplacian is taken from
\cite{FSS7}.

\begin{proposition}
\label{max} Suppose that $p\in C^{1}(\overline{\Omega})$, $u\in W_{0}%
^{1,p(x)}(\Omega)\setminus\left\{  0\right\}  $ and $u\geq0$ in $\Omega$. If
$-\Delta_{p(x)}u\geq0$ in $\Omega$ then $u>0$ in $\Omega$.
\end{proposition}

We recall that the inequality $-\Delta_{p(x)}u\geq0,$ for a function $u\in
W_{0}^{1,p(x)}(\Omega),$ means
\[
\int_{\Omega}\left\vert \nabla u\right\vert ^{p(x)-2}\nabla u\cdot\nabla
\eta\,\mathrm{d}x\geq0,\quad\forall\,\eta\in W_{0}^{1,p(x)}(\Omega
),\,\mathrm{with}\,\eta\geq0.
\]

Theoretical results involving operators with variable exponent can be found
among the papers \cite{AB, DH11, XF07, FSS7, XF09, FS01, FE03, FZ99, FZ01,
BSS, FL13, MR,MO10, PS14, Rad1,Rad2, St} and in the references therein. For
applications in rheology and image restoration we refer the reader to
\cite{Ac, Ant, Ru} and \cite{CL, Chen}, respectively.

\section{The minimization problem}

In this section we will consider $p\in C^{1}(\overline{\Omega})$ and $q\in
C(\overline{\Omega}),$ with $1\leq q\ll p^{\ast}.$ For practical purposes, $X$
will denote the Sobolev space $W_{0}^{1,p(x)}(\Omega)$ and $k,K:X\rightarrow
\mathbb{R}$ will denote, respectively, the functionals%
\begin{equation}
k(u):=\left\Vert u\right\Vert _{q(x)}\quad\mathrm{and}\quad K(u):=\left\Vert
\nabla u\right\Vert _{p(x)},\quad u\in X. \label{Kk}%
\end{equation}
Since $K(u)=\left\Vert u\right\Vert _{X},$ the functional $K$ is sequentially
weakly lower semicontinuous in $X.$

We will also consider
\begin{equation}
\Lambda_{1}:=\inf_{v\in X\setminus\{0\}}\frac{K(v)}{k(v)} \label{g}%
\end{equation}
which is positive number, according to Proposition \ref{pa}.

We say that a function $u\in X\setminus\left\{  0\right\}  $ is an extremal
function (or minimizer) of $\Lambda_{1}$ if
\[
\frac{K(u)}{k(u)}=\Lambda_{1}.
\]
The next proposition shows that such a function always exists.

\begin{proposition}
\label{p1} There exists a nonnegative extremal function of $\Lambda_{1}.$
\end{proposition}

\begin{proof}
Let $(v_{n})\subset X\setminus\{0\}$ be a minimizing sequence of admissible
functions such that $k(v_{n})=1.$ Thus,
\[
\Lambda_{1}=\lim_{n\rightarrow\infty}K(v_{n}).
\]
Since the sequence $(v_{n})$ is bounded in the reflexive space $X$, there
exist a subsequence $(v_{n_{j}})$ and $u\in X$ such that $v_{n_{j}%
}\rightharpoonup u$ in $X.$ We can assume, from Proposition \ref{pa}, that
$v_{n_{j}}\rightarrow u$ in $L^{q(x)}(\Omega),$ so that $k(v_{n_{j}%
})\rightarrow k(u)=1.$ Since $\left\Vert u\right\Vert _{X}\leq\liminf
\limits_{j}\left\Vert v_{n_{j}}\right\Vert _{X}$ we have
\[
K(u)\leq\lim_{j\rightarrow\infty}K(v_{n_{j}})=\Lambda_{1},
\]
showing thus that $u$ is an extremal function of $\Lambda_{1}$. It is simple
to see that (the nonnegative) function $\left\vert u\right\vert $ is also an
extremal function of $\Lambda_{1}$.
\end{proof}

Our next goal is to derive the Euler-Lagrange equation associated with the
minimizing problem (\ref{g}), which must be satisfied for the extremal
functions of $\Lambda_{1}$. For this we need the following lemma.

\begin{lemma}
(Lemma A.1,\cite{FL13})\label{lemfl} Let $u\in X$ and $\eta\in C_{0}^{\infty
}(\Omega).$ Then,
\[
\left.  \frac{d}{d\varepsilon}K(u+\varepsilon\eta)\right\vert _{\varepsilon
=0}=\frac{\int_{\Omega}\left\vert \frac{\nabla u}{K(u)}\right\vert
^{p(x)-2}\frac{\nabla u}{K(u)}\cdot\nabla\eta\mathrm{d}x}{\int_{\Omega
}\left\vert \frac{\nabla u}{K(u)}\right\vert ^{p(x)}\mathrm{d}x}\,
\]
and
\[
\left.  \frac{d}{d\varepsilon}k(u+\varepsilon\eta)\right\vert _{\varepsilon
=0}=\frac{\int_{\Omega}\left\vert \frac{u}{k(u)}\right\vert ^{q(x)-2}\frac
{u}{k(u)}\eta\mathrm{d}x}{\int_{\Omega}\left\vert \frac{u}{k(u)}\right\vert
^{q(x)}\mathrm{d}x}%
\]

\end{lemma}

We observe that a necessary condition for the inequality
\[
\frac{K(u)}{k(u)}\leq\frac{K(u+\varepsilon\eta)}{k(u+\varepsilon\eta)}%
\]
to hold is that
\[
\left.  \frac{d}{d\varepsilon}\frac{K(u+\varepsilon\eta)}{k(u+\varepsilon
\eta)}\right\vert _{\varepsilon=0}=0,
\]
which can be written as
\[
\frac{1}{K(u)}\left.  \frac{d}{d\varepsilon}K(u+\varepsilon\eta)\right\vert
_{\varepsilon=0}=\frac{1}{k(u)}\left.  \frac{d}{d\varepsilon}k(u+\varepsilon
\eta)\right\vert _{\varepsilon=0}.
\]
Therefore, according to Lemma \ref{lemfl}, if $u$ is an extremal function,
then one must have%
\[
\int_{\Omega}\left\vert \frac{\nabla u}{K(u)}\right\vert ^{p(x)-2}\frac{\nabla
u}{K(u)}\cdot\nabla\eta\mathrm{d}x=\Lambda_{1}S(u)\int_{\Omega}\left\vert
\frac{u}{k(u)}\right\vert ^{q(x)-2}\frac{u}{k(u)}\eta\mathrm{d}x,\quad
\forall\,\eta\in C_{0}^{\infty}(\Omega)
\]
where
\begin{equation}
S(u):=\frac{\int_{\Omega}\left\vert \frac{\nabla u}{K(u)}\right\vert
^{p(x)}\mathrm{d}x}{\int_{\Omega}\left\vert \frac{u}{k(u)}\right\vert
^{q(x)}\mathrm{d}x}. \label{Sv}%
\end{equation}

Hence, since $X$ is the closure of $C_{0}^{\infty}(\Omega)$ in the norm
$\left\Vert \cdot\right\Vert _{X}$, the Euler-Lagrange equation for the
extremal functions is
\begin{equation}
-\operatorname{div}\left(  \left\vert \frac{\nabla u}{K(u)}\right\vert
^{p(x)-2}\frac{\nabla u}{K(u)}\right)  =\Lambda_{1}S(u)\left\vert \frac
{u}{k(u)}\right\vert ^{q(x)-2}\frac{u}{k(u)}. \label{h}%
\end{equation}

\begin{definition}
We say that a real number $\Lambda$ is an \textbf{eigenvalue} if there exists
$u\in X\setminus\left\{  0\right\}  $ such that
\begin{equation}
\int_{\Omega}\left\vert \frac{\nabla u}{K(u)}\right\vert ^{p(x)-2}\frac{\nabla
u}{K(u)}\cdot\nabla\eta\,\mathrm{d}x=\Lambda S(u)\int_{\Omega}\left\vert
\frac{u}{k(u)}\right\vert ^{q(x)-2}\frac{u}{k(u)}\eta\,\mathrm{d}%
x,\quad\forall\,\eta\in X. \label{i}%
\end{equation}
In this case, we say that $u$ is an \textbf{eigenfunction} corresponding to
$\Lambda.$
\end{definition}

\begin{remark}
\label{homo}One can easily verify the following homogeneity property: if $u$
is an eigenfunction corresponding to $\Lambda$ the same holds for $tu,$ for
any $t\in\mathbb{R}\setminus\left\{  0\right\}  .$
\end{remark}

Taking $\eta=u$ in (\ref{i}) and recalling the definition of $S(u)$ in
(\ref{Sv}) we obtain%

\[
K(u)\int_{\Omega}\left\vert \frac{\nabla u}{K(u)}\right\vert ^{p(x)}%
\,\mathrm{d}x=\Lambda S(u)k(u)\int_{\Omega}\left\vert \frac{u}{k(u)}%
\right\vert ^{q(x)}\,\mathrm{d}x=\Lambda k(u)\int_{\Omega}\left\vert
\frac{\nabla u}{K(u)}\right\vert ^{p(x)}\mathrm{d}x,
\]
so that
\[
\Lambda=\frac{K(u)}{k(u)}\geq\Lambda_{1}.
\]
Hence, $\Lambda_{1}$ is called the \textbf{first eigenvalue} and the
corresponding eigenfunctions are called \textbf{first eigenfunctions}.
Clearly, the extremal functions are precisely the first eigenfunctions.

\begin{proposition}
\label{p2} There exists a continuous, strictly positive first eigenfunction.
\end{proposition}

\begin{proof}
Proposition \ref{p1} shows that a nonnegative first eigenfunction $u\in X$
exists, Propositions \ref{sol} and \ref{solab} guarantee that $u\in
C(\overline{\Omega})$ and the strong maximum principle (Proposition \ref{max})
yields that $u>0$ in $\Omega.$
\end{proof}

\begin{remark}
\label{alt} It can be verified that if the norm (\ref{e}) is taken to define%
\[
\widetilde{\Lambda}_{1}:=\inf_{v\in X\setminus\{0\}}\frac{\left\vert \nabla
v\right\vert _{p(x)}}{\left\vert v\right\vert _{q(x)}}%
\]
then
\[
\frac{1}{p^{+}}\widetilde{\Lambda}_{1}\leq\Lambda_{1}\leq q^{+}\widetilde
{\Lambda}_{1}.
\]
Moreover, the same results of Propositions \ref{p1} and \ref{p2} can be
obtained, but associated with an Euler -Lagrange equation a bit more
complicated:
\begin{equation}
\int_{\Omega}p(x)\left\vert \frac{\nabla u}{K_{\star}}\right\vert
^{p(x)-2}\frac{\nabla u}{K_{\star}}.\nabla\eta\,\mathrm{d}x=\frac{K_{\star}%
}{k_{\star}}S_{u}\int_{\Omega}q(x)\left\vert \frac{u}{k_{\star}}\right\vert
^{q(x)-2}\frac{u}{k_{\star}}\eta\,\mathrm{d}x \label{equal}%
\end{equation}
where,%
\[
K_{\star}=\left\vert \nabla u\right\vert _{p(x)},\quad k_{\star}=\left\vert
u\right\vert _{q(x)},\quad S_{u}=\frac{\int_{\Omega}p(x)\left\vert
\frac{\nabla u}{K_{\star}}\right\vert ^{p(x)}\mathrm{d}x}{\int_{\Omega
}q(x)\left\vert \frac{u}{k_{\star}}\right\vert ^{q(x)}\mathrm{d}x}.
\]

Noting that $\frac{p^{-}}{q^{+}}\leq S_{u}\leq\frac{p^{+}}{q^{-}}$ we can show
the existence of a strictly positive eigenfunction. In fact, if $u$ is a
nonnegative function of (\ref{equal}) then, for any $\eta\in W_{0}%
^{1,p(x)}(\Omega)$, with $\eta\geq0$, we have%
\begin{align*}
p^{+}\int_{\Omega}\left\vert \frac{\nabla u}{K_{\star}}\right\vert
^{p(x)-2}\frac{\nabla u}{K_{\star}}\cdot\nabla\eta\,\mathrm{d}x  &  \geq
q^{-}\frac{K_{\star}}{k_{\star}}S_{u}\int_{\Omega}\left\vert \frac{u}%
{k_{\star}}\right\vert ^{q(x)-2}\frac{u}{k_{\star}}\eta\,\mathrm{d}x\\
&  \geq q^{-}\left(  \frac{p^{-}}{q^{+}}\right)  \frac{K_{\star}}{k_{\star}%
}\int_{\Omega}\left\vert \frac{u}{k_{\star}}\right\vert ^{q(x)-2}\frac
{u}{k_{\star}}\eta\,\mathrm{d}x\\
&  \geq\left(  \frac{p^{-}}{q^{+}}\right)  \frac{q^{-}}{q^{+}}\frac
{K(u)}{k(u)}\int_{\Omega}\left\vert \frac{u}{k_{\star}}\right\vert
^{q(x)-2}\frac{u}{k_{\star}}\eta\,\mathrm{d}x.
\end{align*}
It follows that $-\Delta_{p(x)}\left(  \frac{u}{K_{\star}}\right)  \geq0$ what
implies, by Proposition \ref{max}, that $u>0$ in $\Omega$. Thus, we can see
that the use of $p(x)^{-1}\mathrm{d}x$ simplifies the equations a little.
\end{remark}

According to Lemma \ref{lemfl}, the Gateaux derivatives $K^{\prime},k^{\prime
}$ are given, respectively, by%
\[
\left\langle K^{\prime}(u),v\right\rangle =\frac{\int_{\Omega}\left\vert
\frac{\nabla u}{K(u)}\right\vert ^{p(x)-2}\frac{\nabla u}{K(u)}\cdot\nabla
v\,\mathrm{d}x}{\int_{\Omega}\left\vert \frac{\nabla u}{K(u)}\right\vert
^{p(x)}\mathrm{d}x},\quad u,v\in X
\]
and
\[
\left\langle k^{\prime}(u),v\right\rangle =\frac{\int_{\Omega}\left\vert
\frac{u}{k(u)}\right\vert ^{q(x)-2}\frac{u}{k(u)}v\,\mathrm{d}x}{\int_{\Omega
}\left\vert \frac{u}{k(u)}\right\vert ^{q(x)}\mathrm{d}x},\quad u,v\in X.
\]

It is simple to check that $K,k\in C^{1}(X,\mathbb{R})$ (see \cite{XF09,FE03}%
). Thus, we define%
\[
\mathcal{M}:=\{u\in X:k(u)=1\}=k^{-1}(1).
\]
Since $1$ is a regular value of $k$, the set $\mathcal{M}$ is a submanifold of
class $C^{1}$ in $X$. The functional
\[
\widetilde{K}:=K\mid_{\mathcal{M}}:\mathcal{M}\rightarrow\mathbb{R}%
\]
is of class $C^{1}$ and bounded from below in $\mathcal{M}.$

We know that $u$ is a critical point of $\widetilde{K}$ in $\mathcal{M}$ if
there exists $\Lambda\in\mathbb{R}$ such that%
\[
K^{\prime}(u)=\Lambda k^{\prime}(u)\quad\text{\textrm{in}}\quad X^{\ast},
\]
meaning that
\[
\left\langle K^{\prime}(u),v\right\rangle =\Lambda\left\langle k^{\prime
}(u),v\right\rangle ,\quad\forall\,v\in X.
\]
Therefore, if $u$ is critical point of $\widetilde{K}$ then $u$ is solution of
(\ref{i}) with $\Lambda=K(u)/k(u)$.

Now, by adapting arguments of \cite[Lemma 2.3]{PS14}, we show that
$\widetilde{K}$ satisfies the Palais-Smale condition.

\begin{proposition}
\label{p4} $\widetilde{K}$ satisfies the $(PS)_{c}$ condition for all
$c\in\mathbb{R}$, namely, every sequence $(u_{n})\subset\mathcal{M}$ such that
$\widetilde{K}(u_{n})\rightarrow c$ and $\widetilde{K}^{\prime}(u_{n}%
)\rightarrow0$, has a convergent subsequence.
\end{proposition}

\begin{proof}
First, we show that if $u\in X\setminus\{0\}$, then%
\begin{equation}
\left\vert \left\langle K^{\prime}(u),v\right\rangle \right\vert \leq
K(v)\quad\text{\textrm{and}}\quad\left\vert \left\langle k^{\prime
}(u),v\right\rangle \right\vert \leq k(v),\quad\forall\,v\in X. \label{dosi}%
\end{equation}
We assume that $v\not \equiv 0$ (otherwise the equality in (\ref{dosi}) holds
trivially). Then%
\begin{equation}
\left\vert \left\langle k^{\prime}(u),v\right\rangle \right\vert \leq
\frac{\int_{\Omega}\left\vert \frac{u}{k(u)}\right\vert ^{q(x)-1}%
|v|\,\mathrm{d}x}{\int_{\Omega}\left\vert \frac{u}{k(u)}\right\vert
^{q(x)}\mathrm{d}x} \label{aux13}%
\end{equation}
and, by using the Young inequality%
\[
ab\leq\left(  1-\frac{1}{r}\right)  a^{r/r-1}+\frac{1}{r}b^{r},\quad
\forall\,a,b\geq0,\,r>1,
\]
with $a=\left\vert u/k(u)\right\vert ^{q(x)-1}$, $b=\left\vert
v/k(v)\right\vert $, $r=q(x),$ and integrating over $\Omega$, we have%
\begin{equation}
\int_{\Omega}\left\vert \frac{u}{k(u)}\right\vert ^{q(x)-1}\left\vert \frac
{v}{k(v)}\right\vert \,\mathrm{d}x\leq\int_{\Omega}\left\vert \frac{u}%
{k(u)}\right\vert ^{q(x)}\mathrm{d}x-\int_{\Omega}\left\vert \frac{u}%
{k(u)}\right\vert ^{q(x)}\frac{\mathrm{d}x}{q(x)}+\int_{\Omega}\left\vert
\frac{v}{k(v)}\right\vert ^{q(x)}\frac{\mathrm{d}x}{q(x)}. \label{aux12}%
\end{equation}

Since%
\[
\left\Vert u\right\Vert _{q(x)}=k(u)\quad\text{\textrm{and}}\quad\left\Vert
v\right\Vert _{q(x)}=k(v)
\]
it follows from Proposition \ref{modul} (item $a)$) that%
\[
\int_{\Omega}\left\vert \frac{u}{k(u)}\right\vert ^{q(x)}\frac{\mathrm{d}%
x}{q(x)}=1=\int_{\Omega}\left\vert \frac{v}{k(v)}\right\vert ^{q(x)}%
\frac{\mathrm{d}x}{q(x)}.
\]
This implies that (\ref{aux12}) can be rewritten as%
\[
\int_{\Omega}\left\vert \frac{u}{k(u)}\right\vert ^{q(x)-1}\left\vert
v\right\vert \,\mathrm{d}x\leq k(v)\int_{\Omega}\left\vert \frac{u}%
{k(u)}\right\vert ^{q(x)}\mathrm{d}x,
\]
which, in view of (\ref{aux13}), leads to the second inequality in (\ref{dosi}).

The first inequality in (\ref{dosi}) is obtained by using the same arguments.

Now, let $c\in\mathbb{R}$ and take a sequence $(u_{n})\subset\mathcal{M}$ such
that $\widetilde{K}(u_{n})\rightarrow c$ and $\widetilde{K}^{\prime}%
(u_{n})\rightarrow0.$ It follows that $K(u_{n})\rightarrow c$ in $X$ and
\begin{equation}
K^{\prime}(u_{n})-c_{n}k^{\prime}(u_{n})\rightarrow0 \label{aux3}%
\end{equation}
in $X^{\ast},$ for some sequence $\left(  c_{n}\right)  \subset\mathbb{R}.$
Since
\[
\left\langle K^{\prime}(u_{n})-c_{n}k^{\prime}(u_{n}),u_{n}\right\rangle
=\left\langle K^{\prime}(u_{n}),u_{n}\right\rangle -c_{n}\left\langle
k^{\prime}(u_{n}),u_{n}\right\rangle =K(u_{n})-c_{n}%
\]
and%
\begin{align*}
\left\Vert \left\langle K^{\prime}(u_{n})-c_{n}k^{\prime}(u_{n}),u_{n}%
\right\rangle \right\Vert _{X}  &  \leq\left\Vert K^{\prime}(u_{n}%
)-c_{n}k^{\prime}(u_{n})\right\Vert _{X^{\ast}}\left\Vert u_{n}\right\Vert
_{X}\\
&  =\left\Vert K^{\prime}(u_{n})-c_{n}k^{\prime}(u_{n})\right\Vert _{X^{\ast}%
}K(u_{n})\rightarrow0,
\end{align*}
one has $c_{n}\rightarrow c.$

Taking into account that $K(u_{n})=\Vert u_{n}\Vert_{X}$ is bounded and that
$X$ is reflexive and compactly embedded into $L^{q(x)}(\Omega)$, we can select
a subsequence $(u_{n_{j}})$ converging weakly in $X$ and strongly in
$L^{q(x)}(\Omega)$ to a function $u\in X.$ The weak convergence guarantees
that $K(u)\leq\liminf K(u_{n_{j}}).$ Thus, since $X$ is uniformly convex, in
order to conclude that $u_{n_{j}}$ converges to $u$ strongly, it is enough to
verify that
\[
\limsup_{j\rightarrow\infty}K(u_{n_{j}})\leq K(u).
\]

It follows from (\ref{dosi}) that
\[
\left\vert \left\langle k^{\prime}(u_{n_{j}}),u_{n_{j}}-u\right\rangle
\right\vert \leq k(u_{n_{j}}-u)=\left\Vert u_{n_{j}}-u\right\Vert
_{q(x)}\rightarrow0.
\]
Combining this fact, (\ref{aux3}) and the boundedness of both sequences
$\left(  c_{n_{j}}\right)  $ and $\left(  \left\Vert u_{n_{j}}-u\right\Vert
_{X}\right)  $ we conclude that $\left\langle K^{\prime}(u_{n_{j}}),u_{n_{j}%
}-u\right\rangle \rightarrow0.$ Since%
\[
\left\langle K^{\prime}(u_{n_{j}}),u_{n_{j}}-u\right\rangle =K(u_{n_{j}%
})-\left\langle K^{\prime}(u_{n_{j}}),u\right\rangle \geq K(u_{n_{j}})-K(u)
\]
we have%
\[
\limsup_{j\rightarrow\infty}K(u_{n_{j}})\leq\limsup_{j\rightarrow\infty
}\left\langle K^{\prime}(u_{n_{j}}),u_{n_{j}}-u\right\rangle +K(u)=K(u),
\]
what finishes the proof.
\end{proof}

Since $\mathcal{M}$ is a closed symmetric submanifold of class $C^{1}$ in $X$
and $\widetilde{K}\in C^{1}(\mathcal{M},\mathbb{R})$ is even, bounded from
below and satisfies the $(PS)_{c}$ condition, we can define an increasing and
unbounded sequence of eigenvalues, by a minimax scheme. For this, we set%
\[
\Sigma:=\{A\subset X\setminus\{0\}:A\ \mathrm{is\ compact}\,\text{\textrm{and}%
}\,A=-A\}
\]
and
\[
\Sigma_{n}:=\{A\in\Sigma:A\subset\mathcal{M}\quad\text{\textrm{and}}%
\quad\gamma(A)\geq n\},\quad n=1,2,...,
\]
where $\gamma$ is the Krasnoselskii genus.

Let us define%
\[
\lambda_{n}:=\inf_{A\in\Sigma_{n}}\sup_{u\in A}\widetilde{K}(u),\quad n\geq1.
\]
It is known that under the above conditions for $\mathcal{M}$ and
$\widetilde{K}$, we have that $\lambda_{n}$ is a critical value of
$\widetilde{K}$ in $\mathcal{M}$ (see \cite{SL88}, Corollary 4.1). Moreover,
since $\Sigma_{k+1}\subset\Sigma_{k}$, we have $\lambda_{k+1}\geq\lambda_{k},$
and so%
\[
0<\lambda_{1}\leq\lambda_{2}\leq...\leq\lambda_{n}\leq\lambda_{n+1}%
...\rightarrow\infty.
\]
In particular $\lambda_{1}=\inf_{v\in\mathcal{M}}K(v)=\Lambda_{1}$ (this
latter equality is consequence of Remark \ref{homo}).

Let us consider the sets%
\[
Q_{n}=\{u\in\mathcal{M}:\widetilde{K}^{\prime}(u)=0\quad\text{\textrm{and}%
}\quad\widetilde{K}(u)=\lambda_{n}\}
\]
and
\[
A_{p(x),q(x)}=\{\lambda\in\mathbb{R}:\lambda\mathrm{\ is\ an\ eigenvalue}\}.
\]
If $u\in Q_{n},$ for a given $n\geq1,$ there exists $\Lambda\in\mathbb{R}$
such that $K^{\prime}(u)=\Lambda k^{\prime}(u)$ in $X^{\ast}$, i.e.,
$\left\langle K^{\prime}(u),v\right\rangle =\Lambda\left\langle k^{\prime
}(u),v\right\rangle $ for any $v\in X$. Thus, $\Lambda=\frac{K(u)}%
{k(u)}=\lambda_{n}$, so that $(\lambda_{n})_{n\geq1}\subset A_{p(x),q(x)}.$
Since $\lambda_{n}\rightarrow\infty$, we conclude the following.

\begin{proposition}
\label{p5} The set of eigenvalues $A_{p(x),q(x)}$ is non-empty, infinite and
$\sup A_{p(x),q(x)}=+\infty$.
\end{proposition}

\begin{remark}
When $p$ and $q$ are constants, the equation (\ref{h}) reduces to%
\[
-\operatorname{div}\left(  \frac{\left\vert \nabla u\right\vert ^{p-2}\nabla
u}{\left\Vert \nabla u\right\Vert _{p}^{p-1}}\right)  =\Lambda\left(
\frac{\sqrt[p]{p}}{\sqrt[q]{q}}\right)  \frac{\left\vert u\right\vert ^{q-2}%
u}{\left\Vert u\right\Vert _{q}^{q-1}}%
\]
and the first eigenvalue is given by%
\[
\Lambda_{1}=\frac{\sqrt[q]{q}}{\sqrt[p]{p}}\inf_{W_{0}^{1,p}(\Omega
)\setminus\{0\}}\frac{\left\Vert \nabla v\right\Vert _{p}}{\left\Vert
v\right\Vert _{q}}.
\]

\end{remark}

\section{Extremal functions for $\frac{\left\Vert \cdot\right\Vert _{lp(x)}%
}{\left\Vert \cdot\right\Vert _{\infty}}$}

We recall the Morrey inequality, valid for $p>N:$%
\[
\Vert u\Vert_{C^{0,\gamma}(\overline{\Omega})}\leq C\Vert\nabla u\Vert
_{L^{p}(\Omega)},\quad\forall\,u\in W_{0}^{1,p}(\Omega),
\]
where $\gamma:=1-\frac{N}{p}$ and the positive constant $C$ depends only on
$N$, $\Omega$ and $p$. An important consequence of this inequality is the
compactness of the embedding
\[
W_{0}^{1,p}(\Omega)\hookrightarrow C(\overline{\Omega}).
\]

Combining this fact with Proposition \ref{imer}, we can verify the compactness
of the embeddings
\[
W_{0}^{1,lp(x)}(\Omega)\hookrightarrow C(\overline{\Omega})\quad
\mathrm{and}\quad W_{0}^{1,lp(x)}(\Omega)\hookrightarrow L^{jq(x)}(\Omega),
\]
where here, and throughout this section:

\begin{itemize}
\item $l,j\in\mathbb{N},$ with $l\geq N;$

\item $p\in C^{1}(\overline{\Omega}),$ with $1<p^{-}\leq p^{+}<\infty$ (so
that $lp(x)\geq lp^{-}>N$);

\item $q\in C(\overline{\Omega})$ with $1<q^{-}\leq q^{+}<\infty.$
\end{itemize}

The following lemma is proved in \cite{FL13}.

\begin{lemma}
\label{infi} If $u\in L^{\infty}(\Omega)$, then%
\[
\lim_{j\rightarrow\infty}\left\Vert u\right\Vert _{jq(x)}=\left\Vert
u\right\Vert _{\infty}.
\]

\end{lemma}

The previous lemma is also valid if we consider an increasing sequence of
functions $(q_{j})\subset C(\overline{\Omega})$ such that $q_{j}%
\rightarrow\infty$ uniformly.

Let us define%
\[
\Lambda_{l,j}:=\inf\left\{  \frac{\left\Vert \nabla v\right\Vert _{lp(x)}%
}{\left\Vert v\right\Vert _{jq(x)}}:v\in W_{0}^{1,lp(x)}(\Omega)\setminus
\{0\}\right\}
\]
and%
\begin{equation}
\mu_{l}:=\inf\left\{  \frac{\left\Vert \nabla v\right\Vert _{lp(x)}%
}{\left\Vert v\right\Vert _{\infty}}:v\in W_{0}^{1,lp(x)}(\Omega
)\setminus\{0\}\right\}  .\label{mmul}%
\end{equation}

\begin{proposition}
\label{iguan}One has,%
\begin{equation}
\lim_{j\rightarrow\infty}\Lambda_{l,j}=\mu_{l}. \label{qinf}%
\end{equation}

\end{proposition}

\begin{proof}
It follows from Lemma \ref{infi} that%
\[
\limsup_{j\rightarrow\infty}\Lambda_{l,j}\leq\lim_{j\rightarrow\infty}%
\frac{\left\Vert \nabla v\right\Vert _{lp(x)}}{\left\Vert v\right\Vert
_{jq(x)}}=\frac{\left\Vert \nabla v\right\Vert _{lp(x)}}{\left\Vert
v\right\Vert _{\infty}},\quad\forall\,v\in W_{0}^{1,lp(x)}(\Omega
)\setminus\{0\}.
\]
Therefore,
\[
\limsup_{j\rightarrow\infty}\Lambda_{l,j}\leq\mu_{l}.
\]

For any $j\geq1$, let $u_{l,j}$ denote the extremal of $\Lambda_{l,j}$, that
is,%
\[
\Lambda_{l,j}=\frac{\left\Vert \nabla u_{l,j}\right\Vert _{lp(x)}}{\left\Vert
u_{l,j}\right\Vert _{jq(x)}}.
\]

It follows from (\ref{infi2}) that
\begin{equation}
\left\Vert u_{l,j}\right\Vert _{jq(x)}\leq|\Omega|^{\alpha_{j}}\left\Vert
u_{l,j}\right\Vert _{\infty},\label{pr53}%
\end{equation}
where
\[
\alpha_{j}:=\left\{
\begin{array}
[c]{lll}%
\frac{1}{jq^{+}} & \mathrm{if} & \left\vert \Omega\right\vert \leq1,\\
\frac{1}{jq^{-}} & \mathrm{if} & \left\vert \Omega\right\vert >1.
\end{array}
\right.
\]

Hence,
\[
\frac{\mu_{l}}{\left\vert \Omega\right\vert ^{\alpha_{j}}}\leq\frac{\left\Vert
\nabla u_{l,j}\right\Vert _{lp(x)}}{\left\vert \Omega\right\vert ^{\alpha_{j}%
}\left\Vert u_{l,j}\right\Vert _{\infty}}\leq\frac{\left\Vert \nabla
u_{l,j}\right\Vert _{lp(x)}}{\left\Vert u_{l,j}\right\Vert _{jq(x)}}%
=\Lambda_{l,j}%
\]
and by making $j\rightarrow\infty$ we obtain%
\[
\mu_{l}\leq\liminf_{j\rightarrow\infty}\Lambda_{l,j},
\]
concluding thus the proof of (\ref{qinf}).
\end{proof}

We say that $u\in W_{0}^{1,lp(x)}(\Omega)$ is an extremal function of $\mu
_{l}$ if%
\[
\mu_{l}=\frac{\left\Vert \nabla u\right\Vert _{lp(x)}}{\left\Vert u\right\Vert
_{\infty}}.
\]

\begin{proposition}
Let $l\geq N$ be fixed. There exists $j_{m}\rightarrow\infty$ and a function
$w_{l}\in W_{0}^{1,lp(x)}(\Omega)\cap C(\overline{\Omega})$ such that
$u_{l,j_{m}}\rightarrow w_{l}$ strongly in $C(\overline{\Omega})$ and also in
$W_{0}^{1,lp(x)}(\Omega)$. Moreover, $w_{l}$ is an extremal function of
$\mu_{l}$.
\end{proposition}

\begin{proof}
Let $u_{l,j}$ denote the extremal function of $\Lambda_{l,j}.$ Without loss of
generality we assume that $\left\Vert u_{l,j}\right\Vert _{jq(x)}=1$. Since
the sequence $(u_{l,j})_{j\geq1}$ is uniformly bounded in $W_{0}%
^{1,lp(x)}(\Omega)$, there exist $j_{m}\rightarrow\infty$ and $w_{l}\in
W_{0}^{1,lp(x)}(\Omega)\subset C(\overline{\Omega})$ such that $u_{l,j_{m}}$
converges to $w_{l},$ weakly in $W_{0}^{1,lp(x)}(\Omega)$ and strongly in
$C(\overline{\Omega})$. It follows from (\ref{pr53}) that%
\[
1=\lim_{m\rightarrow\infty}\left\Vert u_{l,j_{m}}\right\Vert _{jq(x)}\leq
\lim_{m\rightarrow\infty}\left\Vert u_{l,j_{m}}\right\Vert _{\infty
}=\left\Vert w_{l}\right\Vert _{\infty},
\]
so that%
\[
\mu_{l}\leq\frac{\left\Vert \nabla w_{l}\right\Vert _{lp(x)}}{\left\Vert
w_{l}\right\Vert _{\infty}}\leq\left\Vert \nabla w_{l}\right\Vert _{lp(x)}%
\leq\lim_{m\rightarrow\infty}\left\Vert \nabla u_{l,j_{m}}\right\Vert
_{lp(x)}=\lim_{m\rightarrow\infty}\Lambda_{l,j_{m}}=\mu_{l}.
\]
Hence,%
\[
\left\Vert w_{l}\right\Vert _{\infty}=1,\quad\mu_{l}=\left\Vert \nabla
w_{l}\right\Vert _{lp(x)}=\lim_{m\rightarrow\infty}\left\Vert \nabla
u_{l,j_{m}}\right\Vert _{lp(x)},
\]
implying that $w_{l}$ is an extremal function of $\mu_{l}$ and that
$u_{l,j_{m}}\rightarrow w_{l}$ strongly in $W_{0}^{1,lp(x)}(\Omega).$
\end{proof}

Now, by adapting arguments of \cite{HL16} we characterize of the extremal
functions of $\mu_{l}$. For this, let us denote by $\Gamma_{u}$ the set of the
points where a function $u\in C(\overline{\Omega})$ assumes its uniform norm,
that is%
\[
\Gamma_{u}:=\{x\in\overline{\Omega}:\left\vert u(x)\right\vert =\left\Vert
u\right\Vert _{\infty}\}.
\]

\begin{lemma}
\label{id55}Let $u,$ $\eta\in C(\overline{\Omega})$, with $u\not \equiv 0$.
One has%
\[
\lim_{\epsilon\rightarrow0^{+}}\frac{\left\Vert u+\epsilon\eta\right\Vert
_{\infty}-\left\Vert u\right\Vert _{\infty}}{\epsilon}=\max
\{\operatorname{sgn}(u(x))\eta(x):x\in\Gamma_{u}\}.
\]

\end{lemma}

\begin{proof}
Let $r>1,$ $\delta>0$ and $t\in\mathbb{R}.$ Since the function $s\mapsto
\left\vert s\right\vert ^{r-2}s$ is increasing we have%
\[
\frac{\left\vert t+\delta\right\vert ^{r}}{r}=%
{\displaystyle\int_{0}^{t+\delta}}
\left\vert s\right\vert ^{r-2}s\mathrm{d}s=\frac{\left\vert t\right\vert ^{r}%
}{r}+%
{\displaystyle\int_{t}^{t+\delta}}
\left\vert s\right\vert ^{r-2}s\mathrm{d}s\geq\frac{\left\vert t\right\vert
^{r}}{r}+\left\vert t\right\vert ^{r-2}t\delta.
\]
Thus, for $x_{0}\in\Gamma_{u},$ $\eta\in C(\overline{\Omega})$ and
$\epsilon>0,$ we obtain
\[
\frac{\left\Vert u+\epsilon\eta\right\Vert _{\infty}^{r}}{r}\geq
\frac{\left\vert u(x_{0})+\epsilon\eta(x_{0})\right\vert ^{r}}{r}\geq
\frac{|u(x_{0})|^{r}}{r}+|u(x_{0})|^{r-2}u(x_{0})\epsilon\eta(x_{0}).
\]

Making $r\rightarrow1^{+}$ (and using that $\left\vert u(x_{0})\right\vert
=\left\Vert u\right\Vert _{\infty}\not =0$) we arrive at the inequality
\[
\frac{\left\Vert u+\epsilon\eta\right\Vert _{\infty}-\left\Vert u\right\Vert
_{\infty}}{\epsilon}\geq\operatorname{sgn}(u(x_{0}))\eta(x_{0}),
\]
which, in view of the arbitrariness of $x_{0}\in\Gamma_{u},$ implies that%
\[
\liminf_{\epsilon\rightarrow0^{+}}\frac{\left\Vert u+\epsilon\eta\right\Vert
_{\infty}-\left\Vert u\right\Vert _{\infty}}{\epsilon}\geq\max
\{\operatorname{sgn}(u(x))\eta(x):x\in\Gamma_{u}\}.
\]

In order to conclude this proof we will obtain the reverse inequality for
$\limsup_{\epsilon\rightarrow0^{+}}.$ For this, we take $\epsilon
_{m}\rightarrow0^{+}$ such that
\[
\limsup_{\epsilon\rightarrow0^{+}}\frac{\left\Vert u+\epsilon\eta\right\Vert
_{\infty}-\left\Vert u\right\Vert _{\infty}}{\epsilon}=\limsup_{m\rightarrow
\infty}\frac{\left\Vert u+\epsilon_{m}\eta\right\Vert _{\infty}-\left\Vert
u\right\Vert _{\infty}}{\epsilon_{m}}%
\]
and select a sequence $\left(  x_{m}\right)  \subset\overline{\Omega}$
satisfying
\[
\left\vert u(x_{m})+\epsilon_{m}\eta(x_{m})\right\vert =\left\Vert
u+\epsilon_{m}\eta\right\Vert _{\infty}.
\]
We can assume (by passing to a subsequence, if necessary) that $x_{m}%
\rightarrow x_{0}\in\overline{\Omega}.$ Of course, $x_{0}\in\Gamma_{u}$ since
$u+\epsilon_{m}\eta\rightarrow u$ in $C(\overline{\Omega}).$

Since $\left\Vert u\right\Vert _{\infty}\geq\left\vert u(x_{m})\right\vert $
we have%
\[
\limsup_{m\rightarrow\infty}\frac{\Vert u+\epsilon_{m}\eta\Vert_{\infty}-\Vert
u\Vert_{\infty}}{\epsilon_{m}}\leq\limsup_{m\rightarrow\infty}\frac{\left\vert
u(x_{m})+\epsilon_{m}\eta(x_{m})\right\vert -\left\vert u(x_{m})\right\vert
}{\epsilon_{m}}%
\]
and since $u(x_{m})+\epsilon_{m}\eta(x_{m})\rightarrow u(x_{0})$ we have, for
all $m$ large enough,%
\[
\frac{\left\vert u(x_{m})+\epsilon_{m}\eta(x_{m})\right\vert -\left\vert
u(x_{m})\right\vert }{\epsilon_{m}}=\left\{
\begin{array}
[c]{cc}%
\eta(x_{m}), & u(x_{0})>0\\
-\eta(x_{m}), & u(x_{0})<0
\end{array}
\right.  =\operatorname{sgn}(u(x_{0}))\eta(x_{m}).
\]
It follows that
\[
\limsup_{\epsilon\rightarrow0^{+}}\frac{\Vert u+\epsilon\eta\Vert_{\infty
}-\Vert u\Vert_{\infty}}{\epsilon}\leq\operatorname{sgn}(u(x_{0}))\eta
(x_{0})\leq\max\{\operatorname{sgn}(u(x))\eta(x):x\in\Gamma_{u}\}.
\]

\end{proof}

\begin{theorem}
\label{the1} Let $l\geq N$ be fixed. A function $u\in W_{0}^{1,lp(x)}%
(\Omega)\setminus\{0\}$ is extremal of $\mu_{l}$ if, and only if, $\Gamma
_{u}=\left\{  x_{0}\right\}  $ for some $x_{0}\in\Omega$ and%
\begin{equation}
\int_{\Omega}\left\vert \frac{\nabla u}{K_{l}(u)}\right\vert ^{lp(x)-2}%
\frac{\nabla u}{K_{l}(u)}\cdot\nabla\eta\,\mathrm{d}x=\mu_{l}S_{l}%
(u)\operatorname{sgn}(u(x_{0}))\eta(x_{0}),\quad\forall\,\eta\in
W_{0}^{1,lp(x)}(\Omega), \label{extr}%
\end{equation}
where
\[
K_{l}(u):=\left\Vert \nabla u\right\Vert _{lp(x)}\quad\text{\textrm{and}}\quad
S_{l}(u):=%
{\displaystyle\int_{\Omega}}
\left\vert \frac{\nabla u}{K_{l}(u)}\right\vert ^{lp(x)}\mathrm{d}x.
\]

\end{theorem}

\begin{proof}
Let $u\in W_{0}^{1,lp(x)}(\Omega)\setminus\{0\}$ be an extremal function of
$\mu_{l}$ and fix $\eta\in W_{0}^{1,lp(x)}(\Omega).$ Then%
\[
\mu_{l}=\frac{\left\Vert \nabla u\right\Vert _{lp(x)}}{\left\Vert u\right\Vert
_{\infty}}\leq\frac{\left\Vert \nabla u+\epsilon\nabla\eta\right\Vert
_{lp(x)}}{\left\Vert u+\epsilon\eta\right\Vert _{\infty}},\quad\forall
\,\epsilon>0.
\]
It follows that%
\begin{align*}
0 &  \leq\lim_{\epsilon\rightarrow0^{+}}\frac{1}{\epsilon}\left(
\frac{\left\Vert \nabla u+\epsilon\nabla\eta\right\Vert _{lp(x)}}{\left\Vert
u+\epsilon\eta\right\Vert _{\infty}}-\frac{\left\Vert \nabla u\right\Vert
_{lp(x)}}{\left\Vert u\right\Vert _{\infty}}\right)  \\
&  =\lim_{\epsilon\rightarrow0^{+}}\frac{1}{\epsilon}\left(  \frac{\left\Vert
\nabla u+\epsilon\nabla\eta\right\Vert _{lp(x)}}{\left\Vert u+\epsilon
\eta\right\Vert _{\infty}}-\frac{\left\Vert \nabla u\right\Vert _{lp(x)}%
}{\left\Vert u+\epsilon\eta\right\Vert _{\infty}}+\frac{\left\Vert \nabla
u\right\Vert _{lp(x)}}{\left\Vert u+\epsilon\eta\right\Vert _{\infty}}%
-\frac{\left\Vert \nabla u\right\Vert _{lp(x)}}{\left\Vert u\right\Vert
_{\infty}}\right)  \\
&  =\frac{1}{\left\Vert u\right\Vert _{\infty}}\left(  \frac{1}{S_{l}(u)}%
\int_{\Omega}\left\vert \frac{\nabla u}{K_{l}(u)}\right\vert ^{lp(x)-2}%
\frac{\nabla u}{K_{l}(u)}\cdot\nabla\eta\,\mathrm{d}x-\frac{\left\Vert \nabla
u\right\Vert _{lp(x)}}{\left\Vert u\right\Vert _{\infty}}\max
\{\operatorname{sgn}(u(x))\eta(x):x\in\Gamma_{u}\}\right)  \\
&  =\frac{1}{\left\Vert u\right\Vert _{\infty}}\left(  \frac{1}{S_{l}(u)}%
\int_{\Omega}\left\vert \frac{\nabla u}{K_{l}(u)}\right\vert ^{lp(x)-2}%
\frac{\nabla u}{K_{l}(u)}\cdot\nabla\eta\,\mathrm{d}x-\mu_{l}\max
\{\operatorname{sgn}(u(x))\eta(x):x\in\Gamma_{u}\}\right)  ,
\end{align*}
where we have used Lemma \ref{lemfl} and Lemma \ref{id55}.

Therefore,%
\begin{equation}
\int_{\Omega}\left\vert \frac{\nabla u}{K_{l}(u)}\right\vert ^{lp(x)-2}%
\frac{\nabla u}{K_{l}(u)}\cdot\nabla\eta\,\mathrm{d}x\geq\mu_{l}S_{l}%
(u)\max\{\operatorname{sgn}(u(x))\eta(x):x\in\Gamma_{u}\}. \label{aux1}%
\end{equation}

Now, by replacing $\eta$ by $-\eta$ in this inequality we obtain%
\begin{equation}
\int_{\Omega}\left\vert \frac{\nabla u}{K_{l}(u)}\right\vert ^{lp(x)-2}%
\frac{\nabla u}{K_{l}(u)}\cdot\nabla\eta\,\mathrm{d}x\leq\mu_{l}S_{l}%
(u)\min\{\operatorname{sgn}(u(x))\eta(x):x\in\Gamma_{u}\}. \label{aux2}%
\end{equation}

We then conclude from (\ref{aux1}) and (\ref{aux2}) that%
\begin{align*}
\mu_{l}S_{l}(u)\min\left\{  \operatorname{sgn}(u(x))\eta(x):x\in\Gamma
_{u}\right\}    & =\int_{\Omega}\left\vert \frac{\nabla u}{K_{l}%
(u)}\right\vert ^{lp(x)-2}\frac{\nabla u}{K_{l}(u)}\cdot\nabla\eta
\,\mathrm{d}x\\
& =\mu_{l}S_{l}(u)\max\left\{  \operatorname{sgn}(u(x))\eta(x):x\in\Gamma
_{u}\right\}  .
\end{align*}
Taking into account the arbitrariness of $\eta\in W_{0}^{1,lp(x)}(\Omega)$
this implies that $\Gamma_{u}=\left\{  x_{0}\right\}  $ for some $x_{0}%
\in\Omega.$ Consequently, $u$ satisfies (\ref{extr}) for $x_{0}.$

Reciprocally, if $u\in W_{0}^{1,lp(x)}(\Omega)\setminus\{0\}$ is such that
$\Gamma_{u}=\left\{  x_{0}\right\}  $ for some $x_{0}\in\Omega$ and,
additionally, satisfies (\ref{extr}) for this point, we can choose $\eta=u$ in
(\ref{extr}) to get%
\[
\int_{\Omega}\left\vert \frac{\nabla u}{K_{l}(u)}\right\vert ^{lp(x)-2}%
\frac{\nabla u}{K_{l}(u)}\cdot\nabla u\,\mathrm{d}x=\mu_{l}S_{l}%
(u)\operatorname{sgn}(u(x_{0}))u(x_{0})=\mu_{l}S_{l}(u)\left\Vert u\right\Vert
_{\infty},
\]
so that%
\[
\mu_{l}=\frac{\left\Vert \nabla u\right\Vert _{lp(x)}}{\left\Vert u\right\Vert
_{\infty}}.
\]

\end{proof}

\begin{corollary}
Extremal functions of $\mu_{l}$ do not change sign in $\Omega$.
\end{corollary}

\begin{proof}
Let $u\in W_{0}^{1,lp(x)}(\Omega)\setminus\{0\}$ be an extremal function of
$\mu_{l}$ and $x_{0}\in\Omega$ the only point where $u$ achieves its uniform
norm. If $u(x_{0})>0,$ Theorem \ref{the1} yields
\[
\int_{\Omega}\left\vert \frac{\nabla u}{K_{l}(u)}\right\vert ^{lp(x)-2}%
\frac{\nabla u}{K_{l}(u)}\cdot\nabla\eta\,\mathrm{d}x=\mu_{l}S_{l}%
(u)\eta(x_{0})\geq0,
\]
for all nonnegative $\eta\in W_{0}^{1,lp(x)}(\Omega).$ Proposition \ref{max}
then implies that $u>0$ in $\Omega.$ If $u(x_{0})<0$ we repeat the argument
for the extremal function $-u.$
\end{proof}

We can say that
\begin{equation}
-\operatorname{div}\left(  \left\vert \frac{\nabla u}{K_{l}(u)}\right\vert
^{lp(x)-2}\frac{\nabla u}{K_{l}(u)}\right)  =\mu_{l}S_{l}(u)\operatorname{sgn}%
(u(x_{0}))\delta_{x_{0}}, \label{nov}%
\end{equation}
is the Euler-Lagrange equation associated with the minimization problem
defined by (\ref{mmul}), where $\delta_{x_{0}}$ is the Dirac delta function
concentrated in $x_{0}$. We recall that $\delta_{x_{0}}$ is defined by%
\[
\delta_{x_{0}}(\eta)=\eta(x_{0}),\quad\forall\,\eta\in W_{0}^{1,lp(x)}%
(\Omega).
\]
Thus, the extremal functions of $\mu_{l}$ are precisely the weak solutions of
(\ref{nov}) in the sense of (\ref{extr}).

\begin{remark}
Consider a function $v\in W_{0}^{1,lp(x)}(\Omega)$ such that $\left\vert
v(x_{0})\right\vert =\Vert v\Vert_{\infty}$ for some $x_{0}\in\Omega$ and
suppose that this function satisfies the equation%
\[
\int_{\Omega}\left\vert \frac{\nabla v}{K_{l}(v)}\right\vert ^{lp(x)-2}%
\frac{\nabla v}{K_{l}(v)}\cdot\nabla\eta\,\mathrm{d}x=\mu\left(  \int_{\Omega
}\left\vert \frac{\nabla v}{K_{l}(v)}\right\vert ^{lp(x)}\mathrm{d}x\right)
\operatorname{sgn}(v(x_{0}))\eta(x_{0}),\quad\forall\,\eta\in W_{0}%
^{1,lp(x)}(\Omega)
\]
where $\mu\in\mathbb{R}$. By making $\eta=v$, it follows that%
\[
\mu=\frac{K_{l}(v)}{\left\vert v(x_{0})\right\vert }=\frac{\left\Vert \nabla
v\right\Vert _{lp(x)}}{\left\Vert v\right\Vert _{\infty}}\geq\mu_{l}.
\]
Thus, $\mu_{l}$ can be interpreted as the first eigenvalue of (\ref{extr}).
Moreover, for a given natural $j\geq1$, we know, from Section 3, that there
exists a sequence
\[
0<\lambda_{1}^{l,j}\leq\lambda_{2}^{l,j}\leq\cdots\leq\lambda_{n}^{l,j}%
\leq\lambda_{n+1}^{l,j}\leq\cdots
\]
of eigenvalues, where the exponent functions, in this case, are $lp(x)$ and
$jq(x).$ Proposition \ref{iguan} then says that%
\[
\lim_{j\rightarrow\infty}\lambda_{1}^{l,j}=\mu_{l}.
\]

\end{remark}

\section{The limit problem as $l\rightarrow\infty$}

In this section we maintain $p\in C^{1}(\overline{\Omega}),$ with $1<p^{-}\leq
p^{+}<\infty.$ For each natural $l\geq N$ we denote by $w_{l}$ a positive,
$L^{\infty}$-normalized extremal function of $\mu_{l}.$ Thus,%

\[
w_{l}\in W_{0}^{1,lp(x)}(\Omega),\mathrm{\quad}\left\Vert w_{l}\right\Vert
_{\infty}=1,\mathrm{\quad}w_{l}>0\mathrm{\quad in\,}\Omega,
\]
and%
\[
\mu_{l}=\Vert\nabla w_{l}\Vert_{lp(x)}\leq\frac{\Vert\nabla v\Vert_{lp(x)}%
}{\Vert v\Vert_{\infty}},\quad\forall\,v\in W_{0}^{1,lp(x)}(\Omega
)\setminus\{0\}.
\]

We will also denote by $x_{0}^{l}$ the only maximum point of $w_{l}.$
According to the previous section, $w_{l}$ satisfies
\[
-\Delta_{lp(x)}\left(  \frac{w_{l}}{K_{l}(w_{l})}\right)  =\mu_{l}S_{l}%
(w_{l})\delta_{x_{0}^{l}}\quad\mathrm{in}\ \quad\Omega,
\]
where $\delta_{x_{0}^{l}}$ is the Dirac delta function concentrated in
$x_{0}^{l},$%
\[
K_{l}(w_{l})=\Vert\nabla w_{l}\Vert_{lp(x)}=\mu_{l}\mathrm{\quad and\quad
}S_{l}(w_{l}):=\int_{\Omega}\left\vert \frac{\nabla w_{l}}{K_{l}(w_{l}%
)}\right\vert ^{lp(x)}\mathrm{d}x.
\]
Hence,
\[
\int_{\Omega}\left\vert \frac{\nabla w_{l}}{K_{l}(w_{l})}\right\vert
^{lp(x)-2}\frac{\nabla w_{l}}{K_{l}(w_{l})}\cdot\nabla\eta\,\mathrm{d}%
x=\mu_{l}S_{l}(w_{l})\eta(x_{0}^{l}),\quad\forall\,\eta\in W_{0}%
^{1,lp(x)}(\Omega).
\]

Let us define
\[
\Lambda_{\infty}:=\inf\left\{  \frac{\left\Vert \nabla v\right\Vert _{\infty}%
}{\left\Vert v\right\Vert _{\infty}}:v\in W_{0}^{1,\infty}(\Omega
)\setminus\{0\}\right\}  .
\]
It is a well-known fact that
\[
\Lambda_{\infty}=\frac{\left\Vert \nabla d\right\Vert _{\infty}}{\left\Vert
d\right\Vert _{\infty}}=\frac{1}{\left\Vert d\right\Vert _{\infty}},
\]
where $d$ denotes the distance function to the boundary $\partial\Omega$,
defined by%
\[
d(x):=\inf_{y\in\partial\Omega}\left\vert x-y\right\vert ,\quad x\in
\overline{\Omega}.
\]
We recall that
\[
d\in W_{0}^{1,\infty}(\Omega)\mathrm{\quad and}\quad\left\vert \nabla
d\right\vert =1\text{\ }\mathrm{a.e.}\text{\textrm{\ }}\mathrm{in}%
\text{\textrm{\ }}\Omega.
\]

\begin{lemma}
\label{ineq}Let $\alpha:=\int_{\Omega}\frac{\mathrm{d}x}{p(x)}$ and
$e:=\exp(1).$ If $\alpha e<m<l$, then%
\[
\left\Vert u\right\Vert _{mp(x)}\leq\left\Vert u\right\Vert _{lp(x)}%
\quad\forall\,u\in L^{lp(x)}(\Omega).
\]

\end{lemma}

\begin{proof}
When $u\equiv0$ the equality holds trivially in the above inequality. Thus, we
fix $u\in L^{lp(x)}(\Omega)\setminus\left\{  0\right\}  $ and denote the
modular functions associated to $L^{mp(x)}(\Omega)$ and $L^{lp(x)}(\Omega)$ by
$\rho_{m}$ and $\rho_{l}$, respectively.

By H\"{o}lder's inequality%
\begin{align*}
\rho_{m}(u)  &  =\int_{\Omega}\left\vert u\right\vert ^{mp(x)}\frac
{\mathrm{d}x}{mp(x)}\\
&  =\frac{1}{m}\int_{\Omega}\frac{\left\vert u\right\vert ^{mp(x)}}%
{p(x)^{m/l}}\left(  \frac{1}{p(x)}\right)  ^{\frac{l-m}{l}}\mathrm{d}x\\
&  \leq\frac{1}{m}\left(  \int_{\Omega}\frac{\left\vert u\right\vert ^{lp(x)}%
}{p(x)}\right)  ^{\frac{m}{l}}\left(  \int_{\Omega}\frac{\mathrm{d}x}%
{p(x)}\right)  ^{1-\frac{m}{l}}=\frac{1}{m}\left(  l\rho_{l}(u)\right)
^{\frac{m}{l}}\alpha^{1-\frac{m}{l}}=\left(  \frac{f(l)}{f(m)}\rho
_{l}(u)^{\frac{1}{l}}\right)  ^{m},
\end{align*}
where $f(s)=(\frac{s}{\alpha})^{\frac{1}{s}}.$ Since $f$ is decreasing in
$(\alpha e,\infty)$ and $\alpha e<m<l$ we have
\[
\rho_{m}(u)^{\frac{1}{m}}\leq\frac{f(l)}{f(m)}\rho_{l}(u)^{\frac{1}{l}}%
\leq\rho_{l}(u)^{\frac{1}{l}}.
\]

Hence, by taking $a=\left\Vert u\right\Vert _{lp(x)}\neq0$ and applying item
$\mathrm{b)}$ of Proposition \ref{modul} we conclude that
\[
\rho_{m}\left(  \frac{u}{a}\right)  ^{\frac{1}{m}}\leq\rho_{l}\left(  \frac
{u}{a}\right)  ^{\frac{1}{l}}=1
\]
and then that $\left\Vert \frac{u}{a}\right\Vert _{mp(x)}\leq1.$ This implies
that%
\[
\left\Vert u\right\Vert _{mp(x)}\leq a=\left\Vert u\right\Vert _{lp(x)}.
\]

\end{proof}

\begin{proposition}
\label{aux}There exists a subsequence of $\left(  w_{l}\right)  _{l\in
\mathbb{N}}$ converging strongly in $C(\overline{\Omega})$ to a nonnegative
function $w_{\infty}\in W_{0}^{1,\infty}(\Omega)\setminus\left\{  0\right\}  $
such that%
\begin{equation}
\lim_{l\rightarrow\infty}\mu_{l}=\Lambda_{\infty}=\left\Vert \nabla w_{\infty
}\right\Vert _{\infty}. \label{aux8}%
\end{equation}
Moreover,
\begin{equation}
0\leq w_{\infty}(x)\leq\frac{d(x)}{\left\Vert d\right\Vert _{\infty}%
},\mathrm{\quad for\,almost\,every\,}x\in\Omega. \label{aux9}%
\end{equation}

\end{proposition}

\begin{proof}
Since%
\begin{equation}
\Vert\nabla w_{l}\Vert_{lp(x)}=\mu_{l}\leq\frac{\left\Vert \nabla d\right\Vert
_{lp(x)}}{\left\Vert d\right\Vert _{\infty}} \label{aux4}%
\end{equation}
we can apply Lemma \ref{infi} to get%
\begin{equation}
\limsup_{l\rightarrow\infty}\mu_{l}\leq\frac{\left\Vert \nabla d\right\Vert
_{\infty}}{\left\Vert d\right\Vert _{\infty}}=\Lambda_{\infty}. \label{aux7}%
\end{equation}

Let us take a natural $m>\alpha e,$ where $\alpha$ is given by Lemma
\ref{ineq}, and a subsequence $\left(  \mu_{l_{n}}\right)  _{n\in\mathbb{N}}$
such that
\[
\lim_{n\rightarrow\infty}\mu_{l_{n}}=\liminf_{l\rightarrow\infty}\mu_{l}.
\]

Combining Lemma \ref{ineq} with (\ref{aux4}) we conclude that the sequence
$(w_{l_{n}})_{l_{n}>m}$ is bounded in $W_{0}^{1,mp(x)}(\Omega),$ since
\[
\limsup_{n\rightarrow\infty}\left\Vert \nabla w_{l_{n}}\right\Vert
_{mp(x)}\leq\limsup_{n\rightarrow\infty}\left\Vert \nabla w_{l_{n}}\right\Vert
_{l_{n}p(x)}\leq\Lambda_{\infty}.
\]
Thus, up to a subsequence, we can assume that there exists $w_{\infty}\in
W_{0}^{1,mp(x)}(\Omega)$, such that $w_{l_{n}}$ converges to $w_{\infty},$
weakly in $W_{0}^{1,mp(x)}(\Omega)$ and uniformly in $\overline{\Omega}.$

The uniform convergence, implies that $\left\Vert w_{\infty}\right\Vert
_{\infty}=1$ (since $\left\Vert w_{l_{n}}\right\Vert _{\infty}=1$). The weak
convergence in $W_{0}^{1,mp(x)}(\Omega)$ implies that%
\begin{equation}
\left\Vert \nabla w_{\infty}\right\Vert _{mp(x)}\leq\liminf_{n\rightarrow
\infty}\left\Vert \nabla w_{l_{n}}\right\Vert _{mp(x)}.\label{aux6}%
\end{equation}

Now, applying Lemma \ref{ineq} again, we conclude that%
\[
\liminf_{n\rightarrow\infty}\left\Vert \nabla w_{l_{n}}\right\Vert
_{mp(x)}\leq\liminf_{n\rightarrow\infty}\left\Vert \nabla w_{l_{n}}\right\Vert
_{l_{n}p(x)}=\lim_{n\rightarrow\infty}\mu_{l_{n}}=\liminf_{l\rightarrow\infty
}\mu_{l}.
\]
Hence, (\ref{aux6}) yields
\begin{equation}
\left\Vert \nabla w_{\infty}\right\Vert _{mp(x)}\leq\liminf_{l\rightarrow
\infty}\mu_{l}.\label{aux5}%
\end{equation}

Repeating the above arguments we conclude that $w_{\infty}$ is the weak limit
of a subsequence of $\left(  w_{l_{n}}\right)  _{n\in\mathbb{N}}$ in
$W_{0}^{1,sp(x)}(\Omega)$, for any $s>m.$ This fact implies that $w_{\infty
}\in W_{0}^{1,\infty}(\Omega)$. Then, by making $m\rightarrow\infty$ in
(\ref{aux5}), using Lemma \ref{infi} and (\ref{aux7}) we conclude that%
\[
\Lambda_{\infty}\leq\left\Vert \nabla w_{\infty}\right\Vert _{\infty}%
\leq\liminf_{l\rightarrow\infty}\mu_{l}\leq\limsup_{l\rightarrow\infty}\mu
_{l}\leq\Lambda_{\infty},
\]
which gives (\ref{aux8}).

Since the Lipschitz constant of $w_{\infty}$ is $\left\Vert \nabla w_{\infty
}\right\Vert _{\infty}=\Lambda_{\infty}=\dfrac{1}{\left\Vert d\right\Vert
_{\infty}}$, we have%
\[
\left\Vert d\right\Vert _{\infty}\left\vert w_{\infty}(x)-w_{\infty
}(y)\right\vert \leq\left\vert x-y\right\vert
\]
for almost all $x\in\Omega$ and $y\in\partial\Omega.$ Since $w\equiv0$ on the
boundary $\partial\Omega,$ we obtain
\[
\left\Vert d\right\Vert _{\infty}w_{\infty}(x)\leq\inf_{y\in\partial\Omega
}\left\vert x-y\right\vert =d(x).
\]

\end{proof}

We show in the sequel that the functions $w_{\infty}$ and $d$ have a maximum
point in common, which is obtained as a cluster point of the sequence $\left(
x_{0}^{l}\right)  _{l\in\mathbb{N}}.$

\begin{corollary}
\label{uinf} There exists $x_{\star}\in\Omega$ such that%
\[
w_{\infty}(x_{\star})=\left\Vert w_{\infty}\right\Vert _{\infty}%
=1\quad\mathrm{and}\quad d(x_{\star})=\left\Vert d\right\Vert _{\infty}.
\]

\end{corollary}

\begin{proof}
Let $\left(  w_{l_{n}}\right)  _{n\in\mathbb{N}}$ be a sequence converging
uniformly to $w_{\infty},$ which is given by Proposition \ref{aux}. Up to a
subsequence, we can assume that $x_{0}^{l_{n}}\rightarrow x_{\star}%
\in\overline{\Omega}.$ Since $w_{l_{n}}(x_{0}^{l_{n}})=1=\left\Vert w_{l_{n}%
}\right\Vert _{\infty}$ we have $w_{\infty}(x_{\star})=1=\left\Vert w_{\infty
}\right\Vert _{\infty},$ showing that $x_{\star}\in\Omega.$ The conclusion
stems from (\ref{aux9}), since%
\[
1=w_{\infty}(x_{\star})\leq\frac{d(x_{\star})}{\left\Vert d\right\Vert
_{\infty}}\leq1.
\]

\end{proof}

In the sequel we recall the concept of viscosity solutions for an equation of
the form%
\begin{equation}
H(x,u,\nabla u,D^{2}u)=0\quad\mathrm{in}\,D \label{H=0}%
\end{equation}
where $H$ is a partial differential operator of second order and $D$ denotes a
bounded domain of $\mathbb{R}^{N}.$

\begin{definition}
Let $\phi\in C^{2}(D),$ $x_{0}\in D$ and $u\in C(D).$ We say that $\phi$
touches $u$ from below at $x_{0}$ if%
\[
\phi(x_{0})=u(x_{0})\quad\mathrm{and}\quad\phi(x)<u(x),\quad x\neq x_{0}%
\]
Analogously, we say that $\phi$ touches $u$ from above at $x_{0}$ if%
\[
\phi(x_{0})=u(x_{0})\quad\mathrm{and}\quad\phi(x)>u(x),\quad x\neq x_{0}.
\]

\end{definition}

\begin{definition}
\label{subsup} We say that $u\in C(D)$ is a viscosity supersolution of the
equation (\ref{H=0}) if, whenever $\phi\in C^{2}(D)$ touches $u$ from below at
a point $x_{0}\in D$, we have%
\[
H(x_{0},\phi(x_{0}),\nabla\phi(x_{0}),D^{2}\phi(x_{0}))\leq0.
\]
Analogously, we say that $u$ is a viscosity subsolution if, whenever $\psi\in
C^{2}(D)$ touches $u$ from above at a point $x_{0}\in D$, we have%
\[
H(x_{0},\psi(x_{0}),\nabla\psi(x_{0}),D^{2}\psi(x_{0}))\geq0.
\]
And we say that $u$ is a viscosity solution, if $u$ is both a viscosity
supersolution and a viscosity subsolution.
\end{definition}

Note that the differential operator $H$ is evaluated for the test functions
only at the touching point.

In order to interpret the equation
\begin{equation}
\Delta_{p(x)}\left(  \frac{u}{K(u)}\right)  =0\quad\mathrm{in}\quad
\Omega\label{eqfra}%
\end{equation}
in the viscosity sense, we need to find the expression of the corresponding
differential operator $H.$ If $\phi$ is a function of class $C^{2}$, one can
verify that the $p(x)$-Laplacian is given by%
\[
\Delta_{p(x)}\phi=\left\vert \nabla\phi\right\vert ^{p(x)-4}\left\{
\left\vert \nabla\phi\right\vert ^{2}\Delta\phi+(p(x)-2)\Delta_{\infty}%
\phi+\left\vert \nabla\phi\right\vert ^{2}\ln\left\vert \nabla\phi\right\vert
\left\langle \nabla\phi,\nabla p\right\rangle \right\}
\]
where $\left\langle \nabla\phi,\nabla p\right\rangle =\nabla\phi\cdot\nabla p$
and $\Delta_{\infty}$ denotes the $\infty$-Laplacian defined by%
\[
\Delta_{\infty}v:=\frac{1}{2}\langle\nabla v,\nabla\left\vert \nabla
v\right\vert ^{2}\rangle=\sum_{i,j=1}^{N}\frac{\partial v}{\partial x_{i}%
}\frac{\partial v}{\partial x_{j}}\frac{\partial^{2}v}{\partial x_{i}\partial
x_{j}}.
\]

Thus, for a positive constant $t$ one can check that
\begin{equation}
\Delta_{p(x)}(t\phi)=t^{p(x)-1}\left\vert \nabla\phi\right\vert ^{p(x)-4}%
\left\{  \left\vert \nabla\phi\right\vert ^{2}\Delta\phi+(p(x)-2)\Delta
_{\infty}\phi+\left\vert \nabla\phi\right\vert ^{2}\ln(\left\vert \nabla
(t\phi)\right\vert )\left\langle \nabla\phi,\nabla p\right\rangle \right\}
\label{nhom}%
\end{equation}
and by choosing $t=K(u)^{-1}$ we obtain from (\ref{nhom}) that the equation
(\ref{eqfra}) can be rewritten in the form (\ref{H=0}) with
\begin{equation}
H(x,u,\nabla u,D^{2}u):=|\nabla u|^{p(x)-4}\left\{  \left\vert \nabla
u\right\vert ^{2}\Delta\phi+(p(x)-2)\Delta_{\infty}u+|\nabla u|^{2}\ln\left(
\frac{\left\vert \nabla u\right\vert }{K(u)}\right)  \langle\nabla u,\nabla
p\rangle\right\}  ,\label{Hpx}%
\end{equation}
where we are assuming that $p(x)\geq2.$

\begin{proposition}
\label{cara} If $u\in C(\Omega)\cap W_{0}^{1,p(x)}(\Omega)$ is a weak solution
of (\ref{eqfra}) with $K(u)=\left\Vert \nabla u\right\Vert _{p(x)},$ then $u$
is a viscosity solution of this equation.
\end{proposition}

\begin{proof}
We must prove that $u$ is both a viscosity supersolution and a viscosity
subsolution of the equation
\[
H(x,u,\nabla u,D^{2}u)=0\quad\mathrm{in}\,\Omega
\]
for the differential operator $H$ defined by (\ref{Hpx}) with $K(u)=\left\Vert
\nabla u\right\Vert _{p(x)}$. 

By hypothesis, $u$ satisfies%
\begin{equation}
\int_{D}\left\vert \frac{\nabla u}{K(u)}\right\vert ^{p(x)-2}\frac{\nabla
u}{K(u)}\cdot\nabla\eta\,\mathrm{d}x=0,\quad\forall\,\eta\in W_{0}%
^{1,p(x)}(\Omega).\label{phar}%
\end{equation}

Let us prove by contradiction that $u$ is a viscosity supersolution. Thus, we
suppose that there exist $x_{0}\in\Omega$ and $\phi\in C^{2}(\Omega)$ with
$\phi$ touching $u$ from below at $x_{0}$ and satisfying
\[
H(x_{0},\phi(x_{0}),\nabla\phi(x_{0}),D^{2}\phi(x_{0}))>0.
\]
By continuity, there exists a ball $B(x_{0},r)\subset\Omega,$ with $r$ small
enough, such that
\[
H(x,\phi(x),\nabla\phi(x),D^{2}\phi(x))>0,\quad\forall\,x\in B(x_{0},r).
\]
This means that%
\begin{equation}
\Delta_{p(x)}\left(  \frac{\phi}{K(u)}\right)  >0\quad\mathrm{in}%
\,B(x_{0},r).\label{suphar}%
\end{equation}

Let $\varphi:=\phi+\frac{m}{2}$, with $m=\min\limits_{\partial B(x_{0}%
,r)}(u-\phi)>0,$ and take
\[
\eta=(\varphi-u)_{+}\chi_{B(x_{0},r)}\in W_{0}^{1,p(x)}(\Omega).
\]
We have $\varphi<u$ on $\partial B(x_{0},r)$ and $\varphi(x_{0})>u(x_{0})$.

If $\eta\not \equiv 0$, we multiply (\ref{suphar}) by $\eta$ and integrate by
parts to obtain%
\[
\int_{D}\left\vert \frac{\nabla\phi}{K(u)}\right\vert ^{p(x)-2}\frac
{\nabla\phi}{K(u)}\cdot\nabla\eta\,\mathrm{d}x<0.
\]
Note that $\nabla\eta=\nabla\varphi-\nabla u=\nabla\phi-\nabla u$ in the set
$\left\{  \varphi>u\right\}  $ (and $\nabla\eta=0$ in $B(x_{0},r)\cap\left\{
\varphi\leq u\right\}  $). Thus, subtracting (\ref{phar}) from the above
inequality we obtain%
\[
\int_{\left\{  \varphi>u\right\}  }\left\langle \left\vert \frac{\nabla\phi
}{K(u)}\right\vert ^{p(x)-2}\frac{\nabla\phi}{K(u)}-\left\vert \frac{\nabla
u}{K(u)}\right\vert ^{p(x)-2}\frac{\nabla u}{K(u)},\nabla\eta\right\rangle
\mathrm{d}x<0
\]
implying thus that%
\begin{equation}
\int_{\left\{  \varphi>u\right\}  }\left\langle \left\vert \frac{\nabla\phi
}{K(u)}\right\vert ^{p(x)-2}\frac{\nabla\phi}{K(u)}-\left\vert \frac{\nabla
u}{K(u)}\right\vert ^{p(x)-2}\frac{\nabla u}{K(u)},\frac{\nabla\phi}%
{K(u)}-\frac{\nabla u}{K(u)}\right\rangle \mathrm{d}x<0\label{aux10}%
\end{equation}
where the domain of integration is contained in $B(x_{0},r)$. However, it is
well known that%
\[
\langle|Y|^{p-2}Y-|X|^{p-2}X,Y-X\rangle\geq0\quad\forall\,X,Y\in\mathbb{R}%
^{N}\,\mathrm{and\,}p>1.
\]
Thus, by making $Y=\frac{\nabla\phi}{K(u)}$ , $X=\frac{\nabla u}{K(u)}$ and
$p=p(x)$ we see that (\ref{aux10}) cannot occurs. Therefore, $\eta\equiv0.$
But this implies that $\varphi\leq u$ in $B(x_{0},r),$ contradicting
$\varphi(x_{0})>u(x_{0}).$

Analogously, we can show that $u$ is a viscosity subsolution.
\end{proof}

Our next result states that the function $w_{\infty}$ given by Proposition
\ref{aux} is a viscosity solution of the equation
\[
\Delta_{\infty(x)}\left(  \frac{u}{\left\Vert \nabla u\right\Vert _{\infty}%
}\right)  =0
\]
in the punctured domain $D=\Omega\setminus\{x_{\star}\},$ where $x_{\star}$ is
the maximum point of $w_{\infty}$ given by Corollary \ref{uinf} and
$\Delta_{\infty(x)}$ is the differential operator
\[
\Delta_{\infty(x)}u:=\Delta_{\infty}u+|\nabla u|^{2}\ln|\nabla u|\langle\nabla
u,\nabla\ln p\rangle.
\]
Note that if $t>0,$ then
\begin{equation}
\Delta_{\infty(x)}(\frac{u}{t}):=t^{-3}\left\{  \Delta_{\infty}u+|\nabla
u|^{2}\ln|\nabla(\frac{u}{t})|\langle\nabla u,\nabla\ln p\rangle\right\}  .
\label{aux11}%
\end{equation}

\begin{lemma}
\label{lemita}(\cite{Li15}) Suppose that $f_{n}\rightarrow f$ uniformly in
$\overline{D}$, where $f_{n},\,f\in C(\overline{D})$. If $\phi\in C^{2}(D)$
touches $f$ from below at $x_{0}\in D$, then there exists $x_{j}\rightarrow
x_{0}$ such that%
\[
f_{n_{l}}(x_{j})-\phi(x_{j})=\min_{D}\{f_{n_{j}}-\phi\}.
\]

\end{lemma}

\begin{theorem}
The function $w_{\infty}$ is a viscosity solution of
\begin{equation}
\left\{
\begin{array}
[c]{ll}%
\Delta_{\infty(x)}\left(  \dfrac{u}{\left\Vert \nabla u\right\Vert _{\infty}%
}\right)  =0 & \mathrm{in}\quad D=\Omega\setminus\{x_{\star}\}\\
\dfrac{u}{\left\Vert \nabla u\right\Vert _{\infty}}=d & \mathrm{on}%
\quad\partial D=\partial\Omega\cup\left\{  x_{\star}\right\}  .
\end{array}
\right.  \label{infeq}%
\end{equation}

\end{theorem}

\begin{proof}
Taking into account that $\left\Vert \nabla w_{\infty}\right\Vert _{\infty
}=\Lambda_{\infty}$ we just need to show that $w_{\infty}$ satisfies
\[
\left\{
\begin{array}
[c]{ll}%
\Delta_{\infty(x)}\left(  \dfrac{u}{\Lambda_{\infty}}\right)  =0 &
\mathrm{in}\quad D=\Omega\setminus\{x_{\star}\}\\
\dfrac{u}{\Lambda_{\infty}}=d & \mathrm{on}\quad\partial D=\partial\Omega
\cup\left\{  x_{\star}\right\}
\end{array}
\right.
\]
in the viscosity sense.

Since $\Lambda_{\infty}=\dfrac{1}{\left\Vert d\right\Vert _{\infty}}$, it
follows from Corollary \ref{uinf} that
\[
\frac{w_{\infty}(x_{\star})}{\Lambda_{\infty}}=w_{\infty}(x_{\star})\left\Vert
d\right\Vert _{\infty}=d(x_{\star}).
\]
Thus, taking into account that $w_{\infty}|_{\partial\Omega}=0=d|_{\partial
\Omega},$ we conclude that $\dfrac{w_{\infty}}{\Lambda_{\infty}}=d$ on
$\partial\Omega\cup\{x_{\star}\}.$

In order to show that $w_{\infty}$ is a viscosity supersolution, let $x_{0}%
\in\Omega\setminus\{x_{\star}\}$ and $\phi\in C^{2}(\Omega\setminus\{x_{\star
}\})$ be such that $\phi$ touches $w_{\infty}$ from below at $x_{0}$, i.e.%
\[
\phi(x_{0})=w_{\infty}(x_{0})\quad\mathrm{and}\quad\phi(x)<w_{\infty
}(x),\ \mathrm{for}\ x\neq x_{0}.
\]

We claim that
\[
\Delta_{\infty(x)}\left(  \frac{\phi(x_{0})}{\Lambda_{\infty}}\right)  \leq0,
\]
where the expression of the differential operator is given by (\ref{aux11}).
Since the above inequality holds trivially when $\nabla\phi(x_{0})=0,$ we
assume that $\left\vert \nabla\phi(x_{0})\right\vert \neq0.$ So, let us take a
ball $B_{\epsilon}(x_{0})\subset\Omega\setminus\{x_{\star}\}$ such that%
\[
\left\vert \nabla\phi(x)\right\vert \neq0\quad\mathrm{in}\,B_{\epsilon}%
(x_{0}).
\]

Proposition \ref{aux} and its Corollary \ref{uinf} guarantee the existence of
a subsequence of indexes $\left(  l_{n}\right)  _{n\in\mathbb{N}}$ such that
$w_{l_{n}}\rightarrow w_{\infty}$ in $C(\overline{\Omega})$ and $x_{0}^{l_{n}%
}\rightarrow x_{\star},$ where $x_{0}^{l_{n}}$ denotes a maximum point of
$w_{l_{n}}.$ It follows that $w_{l_{n}}\rightarrow w_{\infty}$ uniformly in
$\overline{B_{\epsilon}(x_{0})}$ and $x_{0}^{l_{n}}\notin B_{\epsilon}(x_{0})$
for all $n$ large enough.

Applying Lemma \ref{lemita} to $\overline{B_{\epsilon}(x_{0})}$ we can assume
that (up to pass to another subsequence) there exists $y_{n}\rightarrow x_{0}$
such that%
\[
m_{n}:=\min_{B_{\epsilon}(x_{0})}\{w_{l_{n}}-\phi\}=w_{l_{n}}(y_{n}%
)-\phi(y_{n}).
\]
Thus, the function $\phi_{n}(x):=\phi(x)+m_{n}-|x-y_{n}|^{4},$ which belongs
to $C^{2}(B_{\epsilon}(x_{0})),$ satisfies%
\[
\phi_{n}(y_{n})=\phi(y_{n})+m_{n}=w_{l_{n}}(y_{n})\quad\mathrm{and}\quad
\phi_{n}(x)\leq w_{l_{n}}(x)-|x-y_{n}|^{4}<w_{l_{n}}(x)
\]
for all $x\in B_{\epsilon}(x_{0})\setminus\{y_{n}\}$. That is, $\phi_{n}$
touches $w_{l_{n}}$ from below at $y_{n}$.

Let $H_{l}$ denote the differential operator associated with the equation
$\Delta_{lp(x)}\left(  u/\mu_{l}\right)  =0,$ that is,%
\[
H_{l}(x,u,\nabla u,D^{2}u):=\left\vert \nabla u\right\vert ^{lp(x)-4}\left\{
\left\vert \nabla u\right\vert ^{2}\Delta\phi+(lp(x)-2)\Delta_{\infty
}u+\left\vert \nabla u\right\vert ^{2}\ln\left(  \frac{\left\vert \nabla
u\right\vert }{\mu_{l}}\right)  \langle\nabla u,l\nabla p\rangle\right\}  .
\]
(Recall that $\mu_{l}=\left\Vert \nabla w_{l}\right\Vert _{lp(x)}%
\rightarrow\Lambda_{\infty}$ as $l\rightarrow\infty.$)

It follows from Proposition\ \ref{cara} that $w_{l}$ is a viscosity
(super)solution of the equation%
\[
H_{l}(x,u,\nabla u,D^{2}u)=0.
\]
Hence, taking $\phi_{n}$ as a test function for $w_{l_{n}}$, it follows that%
\[
H_{l_{n}}(y_{n},\phi_{n}(y_{n}),\nabla\phi_{n}(y_{n}),D^{2}\phi_{n}%
(y_{n}))\leq0.
\]
This means that%
\[%
\begin{array}
[c]{rr}%
\left\vert \nabla\phi_{n}(y_{n})\right\vert ^{l_{n}p(x)-4}\left\{  \left\vert
\nabla\phi_{n}(y_{n})\right\vert ^{2}\Delta\phi_{n}(y_{n})+(l_{n}%
p(y_{n})-2)\Delta_{\infty}\phi_{n}(y_{n})\right.   & \\
& \\
+\left.  \left\vert \nabla\phi_{n}(y_{n})\right\vert ^{2}\ln\left(  \left\vert
\nabla\phi_{n}(y_{n})\right\vert /\mu_{l_{n}}\right)  \langle\nabla\phi
_{n}(y_{n}),l_{n}\nabla p(y_{n})\rangle\right\}   & \leq0.
\end{array}
\]

Since%
\[
\nabla\phi_{n}(y_{n})=\nabla\phi(y_{n}),\quad\Delta\phi_{n}(y_{n})=\Delta
\phi(y_{n})\quad\mathrm{and}\quad\Delta_{\infty}\phi_{n}(y_{n})=\Delta
_{\infty}\phi(y_{n})
\]
and $\nabla\phi(y_{n})\neq0$ for $n$ large enough, we have
\[%
\begin{array}
[c]{rr}%
\left\vert \nabla\phi(y_{n})\right\vert ^{l_{n}p(x)-4}\left\{  \left\vert
\nabla\phi(y_{n})\right\vert ^{2}\Delta\phi(y_{n})+(l_{n}p(y_{n}%
)-2)\Delta_{\infty}\phi(y_{n})\right.  & \\
& \\
+\left.  \left\vert \nabla\phi(y_{n})\right\vert ^{2}\ln\left(  \left\vert
\nabla\phi(y_{n})\right\vert /\mu_{l_{n}}\right)  \langle\nabla\phi
(y_{n}),l_{n}\nabla p(y_{n})\rangle\right\}  & \leq0.
\end{array}
\]

Dividing this inequality by $(l_{n}p(y_{n})-2)\left\vert \nabla\phi
(y_{n})\right\vert ^{l_{n}p(x)-4}$, we obtain%
\[
\frac{\left\vert \nabla\phi(y_{n})\right\vert ^{2}\Delta\phi(y_{n})}%
{l_{n}p(y_{n})-2}+\Delta_{\infty}\phi(y_{n})+\left\vert \nabla\phi
(y_{n})\right\vert ^{2}\ln\left(  \frac{\left\vert \nabla\phi(y_{n}%
)\right\vert }{\mu_{l_{n}}}\right)  \left\langle \nabla\phi(y_{n}%
),\frac{\nabla p(y_{n})}{p(y_{n})-2/l_{n}}\right\rangle \leq0.
\]
Then, by making $n\rightarrow\infty$ we arrive at%
\[
\Delta_{\infty}\phi(x_{0})+\left\vert \nabla\phi(x_{0})\right\vert ^{2}%
\ln\left(  \frac{\left\vert \nabla\phi(x_{0})\right\vert }{\Lambda_{\infty}%
}\right)  \left\langle \nabla\phi(x_{0}),\frac{\nabla p(x_{0})}{p(x_{0}%
)}\right\rangle \leq0.
\]
According to (\ref{aux11}) this implies that%
\[
\Delta_{\infty(x)}\left(  \frac{\phi(x_{0})}{\Lambda_{\infty}}\right)  \leq0
\]
and we conclude thus the proof that $w_{\infty}$ is a viscosity supersolution.

Analogously, we can show that if $\psi\in C^{2}(\Omega\setminus\{x_{\star}\})$
touches $w_{\infty}$ from above at the point $x_{0}$, then%
\[
\Delta_{\infty(x)}\left(  \frac{\psi(x_{0})}{\Lambda_{\infty}}\right)  \geq0.
\]

Therefore, $w_{\infty}$ satisfies (\ref{infeq}) in the viscosity sense.
\end{proof}

The following uniqueness result can be found in \cite{Li10}.

\begin{proposition}
(\cite{Li10}, Theorem 1.2)\label{unidir} Let $D$ be a bounded domain of
$\mathbb{R}^{N}$ and $f:\partial D\rightarrow\mathbb{R}$ be a Lipschitz
continuous function. There exists a unique viscosity solution $u\in
C(\overline{D})\cap W^{1,\infty}(D)$ for the Dirichlet boundary value problem%
\[
\left\{
\begin{array}
[c]{ll}%
\Delta_{\infty(x)}u=0 & \quad\mathrm{in\,}D\\
u=f & \quad\mathrm{on\,}\partial D.
\end{array}
\right.
\]

\end{proposition}

It follows from this result, with $f=d$ and $D=\Omega\setminus\{x_{\star}\},$
that $w_{\infty}$ is the only solution of the Dirichlet problem (\ref{infeq}).

Following the ideas of \cite{YY08} and \cite{EP16} we give a condition on
$\Omega$ that leads to the equality $w_{\infty}=d/\left\Vert d\right\Vert
_{\infty}.$ For this we recall that $d\in C^{1}\left(  \Omega\setminus
\mathcal{R}(\Omega)\right)  $ where $\mathcal{R}(\Omega)$ denotes the
\textit{ridge of} $\Omega,$ defined as the set of all points in $\Omega$ whose
distance to the boundary is reached at least at two points (see \cite{BDM,
EH}). Notice that $\mathcal{R}(\Omega)$ contains the maximum points of $d.$
Since $d$ is a viscosity solution of the eikonal equation $\left\vert \nabla
u\right\vert =1$ in $\Omega,$ it is simple to check that $\Delta_{\infty
(x)}d=0$ in $\Omega\setminus\mathcal{R}(\Omega)$ in the viscosity sense.

\begin{proposition}
\label{cond} If $\mathcal{R}(\Omega)$ is a singleton set, then $w_{l}%
\rightarrow\dfrac{d}{\left\Vert d\right\Vert _{\infty}}$ uniformly in
$\overline{\Omega}$ and $x_{0}^{l}\rightarrow x_{\star}.$
\end{proposition}

\begin{proof}
It follows from Corollary \ref{uinf} that $\mathcal{R}(\Omega)=\left\{
x_{\star}\right\}  ,$ since $x_{\star}$ is a maximum point of $d.$ Therefore,
$d\in C^{1}\left(  \Omega\setminus\left\{  x_{\star}\right\}  \right)  $ what
implies that $d$ is a viscosity solution of
\[
\left\{
\begin{array}
[c]{ll}%
\Delta_{\infty(x)}u=0 & \quad\mathrm{in\,}\Omega\setminus\left\{  x_{\star
}\right\} \\
u=d & \quad\mathrm{on\,}\partial\left(  \Omega\setminus\left\{  x_{\star
}\right\}  \right)  =\partial\Omega\cup\left\{  x_{\star}\right\}  .
\end{array}
\right.
\]
Therefore, by the uniqueness stated in Proposition \ref{unidir} (with
$D:=\Omega\setminus\left\{  x_{\star}\right\}  )$ we have $\dfrac{w_{\infty}%
}{\Lambda_{\infty}}=d,$ that is%
\[
w_{\infty}=\frac{d}{\left\Vert d\right\Vert _{\infty}}\quad\mathrm{in\,}%
\overline{\Omega}%
\]

These arguments imply that $d/\left\Vert d\right\Vert _{\infty}$ is the only
limit function of any uniformly convergent subsequence of $\left(
w_{l}\right)  _{l\in N}$ and also that $x_{\star}$ is the only cluster point
of the numerical sequence $\left(  x_{0}^{l}\right)  _{l\in N}.$
\end{proof}

Balls, ellipses and other symmetric sets are examples of domains whose ridge
is a singleton set.

\section{Acknowledgements}

C. O. Alves was partially supported by CNPq/Brazil (304036/2013-7) and
INCT-MAT. G. Ercole was partially supported by CNPq/Brazil (483970/2013-1 and
306590/2014-0) and Fapemig/Brazil (APQ-03372-16).

\end{document}